\documentclass[12pt,reqno]{amsart}

\numberwithin{equation}{section} \numberwithin{figure}{section}
\numberwithin{table}{section}

\usepackage{xy}
\xyoption{matrix} \xyoption{arrow} \xyoption{curve}

\newtheorem{theorem}{Theorem}[section]
\newtheorem{thm}[theorem]{Theorem}
\newtheorem{proposition}[theorem]{Proposition}
\newtheorem{lemma}[theorem]{Lemma}
\newtheorem{corollary}[theorem]{Corollary}

\newtheorem{prop}[theorem]{Proposition}
\newtheorem{cor}[theorem]{Corollary}

\theoremstyle{definition}
\newtheorem{definition}[theorem]{Definition}

\newtheorem{remark}[theorem]{Remark}

 \usepackage {amsfonts}
 \usepackage {longtable}

\DeclareMathOperator{\sys}{{\rm sys}}

\newcommand\M {\operatorname{M}}
\newcommand\Nr {\textrm {Nr}}
\newcommand\Norm {\textrm{N}}
\newcommand\Tr {\textrm {Tr}}

\newcommand\bb {\mathbb}
\newcommand\mcal{\mathcal}
\newcommand\area{\textrm{area}}
\newcommand\qb{{\mathcal Q}_B}
\newcommand\QQ{{\mathcal Q}}
\newcommand\Q{{\mathbb{Q}}}
\newcommand\R{{\mathbb{R}}}
\newcommand\Z{{\mathbb{Z}}}
\newcommand\F{{\mathbb{F}}}
\newcommand\HH{{\mathcal{H}}}
\renewcommand\O{{\mcal{O}}}
\long\def\forget#1\forgotten{{}}
\newcommand\sub{{\,\subseteq\,}}
\newcommand\ideal[1]{{\left<#1\right>}}

\newcommand\Tref[1]{{Theorem~\ref{#1}}}
\newcommand\Lref[1]{{Lemma~\ref{#1}}}
\newcommand\Rref[1]{{Remark~\ref{#1}}}
\newcommand\Pref[1]{{Proposition~\ref{#1}}}

\newcommand\Cref[1]{{Corollary~\ref{#1}}}
\newcommand\Sref[1]{{Section~\ref{#1}}}
\newcommand\eq[1]{{(\ref{#1})}}
\newcommand\GL[1][2]{{\operatorname{GL}_{#1}}}
\newcommand\SL[1][2]{{\operatorname{SL}_{#1}}}
\newcommand\PGL[1][2]{{\operatorname{PGL}_{#1}}}
\newcommand\PSL[1][2]{{\operatorname{PSL}_{#1}}}
\newcommand\co{{\,{:}\,}}
\def\normali{{\lhd}} % triangle without the - sign: for ideals.
\newcommand\disc{{\operatorname{disc}}}
\newcommand\tr{{\operatorname{tr}}} % the reduced trace
\newcommand\dimcol[2]{{[{#1}\!:\!{#2}]}} % Produces nicely spaced [K:F].
\newcommand\divides{{\,|\,}}
\newcommand\lam{{\lambda}}
\newcommand\mul[1]{{#1^{\times}}}
\newcommand\sg[1]{{\left<{#1}\right>}}
\newcommand\set[1]{{\left\{{#1}\right\}}}
\newcommand\ra{{\rightarrow}}

\def\s{{\sigma}}
\newcommand\sqb[1][]{{\mathbb{P}\qb^{1}\if!#1!{}\else{(#1)}\fi}}

\newcommand\smat[4]{{\left(\begin{array}{cc}#1 & #2 \\ #3 & #4\end{array}\right)}}
\newcommand\abs[1]{{\left|#1\right|}}
\newcommand\magma{{\texttt{magma}}}

\begin{document}

\author{Karin Katz}\author{Mikhail Katz}\author{Michael
M.~Schein}\author{Uzi Vishne}

\address{Department of Mathematics, Bar Ilan University}

\title[Bolza quaternion order and systoles] {Bolza quaternion order
and asymptotics of systoles along congruence subgroups}


\maketitle

%\footnotesize{}

\begin{abstract}
We give a detailed description of the arithmetic Fuchsian group of the
Bolza surface and the associated quaternion order.  This description
enables us to show that the corresponding principal congruence covers
satisfy the bound $\text{sys}(X) > \frac{4}{3}\log g(X)$ on the
systole, where $g$ is the genus.  We also exhibit the Bolza group as a
congruence subgroup, and calculate out a few examples of ``Bolza
twins'' (using \magma{}).  Like the Hurwitz triplets, these correspond
to the factoring of certain rational primes in the ring of integers of
the invariant trace field of the surface.  We exploit random sampling
combined with the Reidemeister-Schreier algorithm as implemented in
\magma{} to generate these surfaces.
\end{abstract}

\tableofcontents

\section{Introduction}

This article pursues several related goals.  First, we seek to clarify
the algebraic underpinnings of the celebrated Bolza curve which turn
out to be more involved than those of the celebrated Klein quartic.
Furthermore, we seek to provide explicit algebraic foundations, in
terms of a quaternion algebra, for calculating out examples of Riemann
surfaces with particularly high systole corresponding to principal
congruence subgroups in the Bolza order.  In an effort to make the
text intelligible to both algebraists and differential geometers, we
sometimes give detailed proofs that could have been shortened if
addressed to a specific expert audience.

In 2007, Katz, Schaps and Vishne \cite{KSV06a} proved a lower bound
for the systole of certain arithmetic Riemann surfaces, improving
earlier results by Buser and Sarnak (1994 \cite[p.~44]{BS94}).
Particularly sharp results were obtained in \cite{KSV06a} and
\cite{KSV2} for Hurwitz surfaces, namely Riemann surfaces with an
automorphism group of the highest possible order in terms of the
genus~$g$, yielding a lower bound
\begin{equation}
\label{43}
\text{sys}(X_g) > \frac{4}{3}\log g
\end{equation}
for principal congruence subgroups corresponding to a suitable Hurwitz
quaternion order defined over $\Q(\cos\frac{2\pi}{7})$.

Makisumi (2013 \cite{Ma}) proved that the multiplicative
constant~$\frac{4}{3}$ in the bound~\eqref{43} is the best possible
asymptotic value for congruence subgroups of arithmetic Fuchsian
groups.  Schmutz Schaller (1998 \cite[Conjecture~1(i), p.~198]{Sc98})
conjectured that a $4/3$ bound is the best possible among all
hyperbolic surfaces.  Additional examples of surfaces whose systoles
are close to the bound
%\footnote{Is this right?}
were recently constructed by Akrout \& Muetzel (2013 \cite{AM13a},
\cite{AM13b}).  The foundations of the subject were established by
Vinberg (1967 \cite{Vin}).

We seek to extend the bound~\eqref{43} to the case of the family of
Riemann surfaces defined by principal congruence subgroups of the
$(3,3,4)$ triangle group corresponding to a quaternion order defined
over~$\Q(\sqrt{2})$, which is closely related to the Bolza surface.

The Fuchsian group of the Bolza surface, which we henceforth
denote~$B$, is arithmetic, being a subgroup of the group of units,
modulo~$\set{\pm 1}$, in an order of the quaternion algebra%
\begin{equation}
\label{12}
D_B = (-3,\sqrt{2}) = K[i,j \,|\, i^2 = -3,\, j^2 = \sqrt{2},\,ji=-ij]
\end{equation}
over the base field $K = \Q(\sqrt{2})$.  The splitting pattern of this
algebra is determined in \Sref{sec:4}.  Let $O_K = \Z[\sqrt{2}]$ be
the ring of integers of~$K$.  This is a principal ideal domain, so
irreducible elements of $O_K$ are prime.

\begin{lemma}
\label{intro1}
The standard order
\[
\operatorname{span}_{O_K}\set{1,i,j,ij}
\]
in the algebra $D_B$ is contained in precisely two maximal orders
$\QQ$ and~$\QQ'$, which are conjugate to each other.
\end{lemma}
We will prove \Lref{intro1} in \Sref{sec:5}.  This lemma is a
workhorse result used in the analysis of maximal orders below.  We let
$\qb = \QQ$.

\begin{theorem}
\label{intro2}
Almost all principal congruence subgroups of the maximal order $\qb$
satisfy the systolic bound~\eqref{43}.
\end{theorem}

This is proved in \Sref{sec:8}, where a more detailed version of the
result is given.  In fact, Theorem~\ref{intro2} is a consequence of the
following more general result. For an order $Q$ in a quaternion algebra $D$, let $Q^1$ be the group of units of $Q$ and let $d$ be the dimension over $\Q$ of the center of $Q$.  We define a constant $\Lambda_{D,Q} \geq 1$ depending on the local ramification pattern (see Section~\ref{sec:8}). Let $X_1$ be the quotient of the hyperbolic plane $\mcal H^2$ modulo the action of $Q^1$.

\begin{proposition}
%\label{81b}
Suppose $2^{3(d{{-1}})}\Lambda_{D,Q} < \frac{4\pi}{\area(X_1)}$.
Then almost all the principal congruence covers of $X_1$
satisfy the bound $\sys > \frac{4}{3}\log g.$%
\end{proposition}
Note that this is stronger than the Buser--Sarnak bound $\sys
X(\Gamma)>\frac{4}{3} \log g(X(\Gamma)) - c(\Gamma_0)$ where the
constant $c(\Gamma_0)$ could be arbitrarily large.  Returning to the
Bolza order, we have the following result.

\begin{theorem}
\label{intro3}
There are elements $\alpha$ and $\beta$ of norm $1$ in the algebra $D_B$
of~\eqref{12} such that $\qb = O_K[\alpha,\beta]$ as an order.  Let $\qb^1 = \sg{\alpha,\beta}$ be the group generated by $\alpha$ and $\beta$.  Then $\qb^1 / \set{\pm 1}$
is isomorphic to the triangle group~$\Delta_{(3,3,4)}$.
\end{theorem}

In \Cref{114} we find that the Bolza group $B$ has index $24$ in $\qb^1/\set{\pm 1}$ and is generated, as a
normal subgroup of $\qb^1/\set{\pm 1}$, by the element
$(\alpha\beta)^2(\alpha^2\beta^2)^2$. The choice of $\alpha$ and $\beta$ implies that~$B$
is contained in the principal congruence subgroup
$\qb^1(\sqrt{2})/\set{\pm 1}$. However, this congruence subgroup has
torsion: it contains an involution closely related to the
hyperelliptic involution of the Bolza surface (see
Section~\ref{nine}).  Working out the ring structure of $\qb/2\qb$, we
are then able to compute the quotient $B\qb^1(2)/\qb^1(2)$ and
obtain the following.
\begin{theorem}\label{T1.5}
The fundamental group $B$ of the Bolza surface is contained strictly
between two principal congruence subgroups as follows:
\[
\qb^1(2)/\set{\pm 1} \,\subset\, B \,\subset\, \qb^1(\sqrt{2})/\set{\pm 1}.
\]
\end{theorem}
This explicit identification of the Bolza group as a (non-principal) congruence subgroup in the maximal order requires a detailed analysis of quotients, and occupies Sections~\ref{ten}--\ref{sec:13}. % \ref{ten}, \ref{nine}, \ref{sec:12}, \ref{sec:13}.
In contrast, the Fuchsian group of the Klein quartic (which is the Hurwitz surface of least genus) does happen to be a principal congruence subgroup in the group of units of the corresponding maximal order; see \cite[Section~4]{KSV2}.

It follows from \Tref{T1.5} that $B$ is a congruence subgroup.
Moreover, we show that~ $\qb^1/\sg{{{-1}},B} \cong \SL[2](\F_3)$,
explaining some of the symmetries of the Bolza surface. The full
symmetry group, $\GL[2](\F_3)$, comes from the embedding of the
triangle group $\Delta_{(3,3,4)}$ in $\Delta_{(2,3,8)}$; see
\Cref{133}.

In the concluding Sections~\ref{twins} through \ref{s16} we present
``twin Bolza'' surfaces corresponding to factorisations of rational
primes~$7$, $17$, $23$, $31$, $41$, $47$, and $71$ as a product of a
pair of algebraic primes in $\Q(\sqrt{2})$.

Recent publications on systoles include Babenko \& Balacheff
\cite{BB15}; Balacheff, Makover \& Parlier \cite{BMP14}; Bulteau
\cite{Bu15}: Katz \& Sabourau \cite{KS15}; Kowalick, Lafont \&
Minemyer \cite{KLM15}; Linowitz \& Meyer \cite{LM15}.

\section{Fuchsian groups and quaternion algebras}

A cocompact Fuchsian group~$\Gamma\subset\PSL(\mathbb R)$ defines a
hyperbolic Riemann surface~$\HH^2/\Gamma$,
denoted~$X_{\Gamma}$, where~$\HH^2$ is the hyperbolic plane.
If~$\Gamma$ is torsion free, the systole~$\sys (X_{\Gamma})$ satisfies
\[
2\cosh\left(\tfrac12\sys(X_\Gamma)\right)=\min_{M}|\textrm{trace}(M)|,
\]%%
or
\begin{equation}
\label{cosh}
\sys(X_{\Gamma})=
\min_{M}2\,\text{arccosh}\left(\tfrac12|\textrm{trace}(M)|\right),
\end{equation}
where~$M$ runs over all the nonidentity elements of~$\Gamma$.  We will
construct families of Fuchsian groups in terms of suitable orders in
quaternion algebras.  Since the traces in the matrix algebra coincide
with reduced traces (see below) in the quaternion algebra, the
information about lengths of closed geodesics, and therefore about
systoles, can be read off directly from the quaternion algebra,
bypassing the traditional presentation in matrices.

Let~$k$ be a finite dimensional field extension of~$\mathbb{Q}$, let~$a,b\in k^*$, and
consider the following associative algebra over~$k$:
\begin{equation}
\label {quaternions} A=k[i,j\,|\, i^{2}=a,j^{2}=b,ji=-ij].
\end{equation}
The algebra~$A$ admits the following decomposition as a~$k$-vector
space:
\[
A=k1\oplus ki\oplus kj\oplus kij\;.
\]
Such an algebra~$A$, which is always simple, is called a quaternion algebra. The
center of~$A$ is precisely~$k$.

\begin{definition}\label{Tr+Nrd}
Let~$x=x_{0}+x_{1}i+x_{2}j+x_{3}ij\in A$. The \emph{conjugate} of~$x$
(under the unique symplectic involution) is~$x^*=x_{0}-x_{1}i-x_{2}j-x_{3}ij$.
The \emph{reduced
trace} of~$x$ is
\[
\Tr_A(x):=x+x^*=2x_{0},
\]
and the \emph{reduced norm} of~$x$ is
\[
\Nr_A(x):=xx^*=x_{0}^{2}-ax_{1}^{2}-bx_{2}^{2}+abx_{3}^{2}.
\]
\end{definition}

\begin{definition}[cf.~Reiner 1975 \cite{Re75}]
\label{2.3}
An \emph{order} of a quaternion algebra~$A$ (over~$k$) is a subring
with unit, which is a finitely generated module over the ring of
integers~$O_{k}\subset k$, and such that its ring of fractions is
equal to~$A$.

\end{definition}

If~$a$ and~$b$ in \eqref{quaternions} are algebraic integers in
$k^{*}$, then the subring ~$\mathcal O\subset A$ defined by
\begin{equation}
\label{order} \mcal O=O_k1+ O_ki+ O_kj+ O_kij
\end{equation}
is an order of~$A$ (see Katok 1992 \cite[p.~119]{Ka92}), although
not every order has this form; a famous example of an order not
having the form \eqref{order} is the Hurwitz order in Hamilton's
quaternion algebra over the rational numbers. Note that in the order
the scalars are taken from the ring of integers~$O_k$; the scalars
are taken from the field~$k$ when passing to the ring of fractions.

\section{The (2,3,8) and~$(3,3,4)$ triangle groups}\label{sec:238}

The Bolza surface can be defined by a subgroup of either the (2,3,8)
or the~$(3,3,4)$ triangle group.  We will study specific Fuchsian groups
arising as congruence subgroups of the arithmetic triangle group of
type~$(3,3,4)$.  First we clarify the relation between the~$(3,3,4)$ and
the (2,3,8) groups.  Let~$\Delta_{(2,3,8)}$ denote the (2,3,8)
triangle group, i.e.
\begin{equation}
\label{delta1'} \Delta_{(2,3,8)}=\langle x,y\mid
x^2=y^3=(xy)^8=1\rangle.
\end{equation}
Let~$h\co \Delta_{(2,3,8)}\rightarrow \Z/2\Z$ be the homomorphism
sending~$x$ to the nontrivial element of $\Z / 2 \Z$ and~$y$ to the identity element.

\begin{lemma}
\label{kerh}
As a subgroup of~$\Delta_{(2,3,8)}$, the kernel of~$h$ is given by
\[
\ker(h)=\langle \alpha,\beta\mid
\alpha^3=\beta^3=(\alpha\beta)^4=1\rangle
\]
where~$\alpha=y$ and~$\beta=xyx$.
\end{lemma}
\begin{proof}
The presentation can be obtained by means of the Reidemeister-Schreier
method, but here is a direct proof.  Note that~$xy^n x=(xyx)^n
=\beta^n$.  Each element~$t\in\ker(h)$ is of one of 4 types:
\begin{enumerate}
\item
\label{case1}
$t=xy^{n_1}xy^{n_2}\cdots xy^{n_k}x$;
\item
$t=y^{n_1}xy^{n_2}\cdots xy^{n_k}x$;
\item
$t=xy^{n_1}xy^{n_2}\cdots xy^{n_k}$;
\item
$t=y^{n_1}xy^{n_2}\cdots xy^{n_k}$,
\end{enumerate}
with an even number of~$x$'s, where all the exponents~$n_i$ are
either~1 or~2.  To show that each element can be expressed in terms of
$\alpha$ and~$\beta$, we argue by induction on the length of the
presentation in terms of~$x$'s and~$y$'s.  Type~(1) is reduced to (a
shorter) type~(2) by noting that~$xy^{n_1}xy^{n_2}\cdots
xy^{n_k}x=\beta^{n_1} y^{n_2}\cdots xy^{n_k}x$.  Type (2) is reduced
to (a shorter) type (1) by noting that~$y^{n_1}xy^{n_2}\cdots
xy^{n_k}x=\alpha^{n_1}xy^{n_2}\cdots xy^{n_k}x$.  Type~(3) is reduced
to type~(4) by noting that~$xy^{n_1}xy^{n_2}\cdots xy^{n_k}=
\beta^{n_1}y^{n_2}\cdots xy^{n_k}$.  Type~(4) is reduced to (a
shorter) type~(3) by noting that~$y^{n_1}xy^{n_2}\cdots xy^{n_k}=
\alpha^{n_1}xy^{n_2}\cdots xy^{n_k}$.

To check the relations on~$\ker(h)$, note that
\begin{itemize}
\item
$\alpha^3=y^3=1$;
\item
$\beta^3=(xyx)^3=xy^3x=xx=1$;
\item
$(\alpha\beta)^4=(yxyx)^4=y(xy)^8y^{{{-1}}}=1$,
\end{itemize}
completing the proof.
\end{proof}

For a finitely generated non-elementary
subgroup~$\Gamma\subset\PSL(\bb R)$, we define~$\Gamma^{(2)}=\langle
t^2:t\in\Gamma\rangle$.

\begin{lemma}
\label{(2,3,8),kerh}
\label{cory}
For~$\Gamma=\Delta_{(2,3,8)}$ we have~$\Gamma^{(2)}=\ker(h)$, and
therefore the group~$\Delta_{(2,3,8)}^{(2)}$ is isomorphic to the
triangle group~$\Delta_{(3,3,4)}$.
\end{lemma}

\begin{proof}
We have~$\alpha=\alpha^4=(\alpha^2)^2$ and similarly for~$\beta$.
Thus~$\ker(h)\subset\Delta_{(2,3,8)}^{(2)}$.  Choosing~$T$ to be the
right-angle hyperbolic triangle with acute angles~$\frac{\pi}{3}$
and~$\frac{\pi}{8}$, we note that the ``double'' of~$T$, namely the
union of~$T$ and its reflection in its (longer) side opposite the
angle~$\frac{\pi}3$, is an isosceles triangle with
angles~$\frac{\pi}3$,~$\frac{\pi}3$, and~$\frac{\pi}4$, proving the
lemma.
\end{proof}

\begin{definition}[{\cite{MR03}}]
\label{itf}
Let~$\Gamma$ be a finitely generated non-elementary subgroup of
$\PSL(\bb R)$. The \emph{invariant trace field of} of~$\Gamma$,
denoted by~$k\Gamma$, is the field~$\bb Q(\textrm{tr}\Gamma^{(2)})$.
\end{definition}

\begin{definition}
For an~$(\ell,m,n)$ triangle group, let
\[
\lambda(\ell,m,n):=4\cos^2\frac\pi\ell+4\cos^2\frac\pi
m+4\cos^2\frac\pi n+8\cos\frac\pi \ell\cos\frac\pi m\cos\frac\pi n-4.
\]
\end{definition}
In particular,~$\lambda(3,3,4)=\sqrt{2}.$ Therefore by
\cite[p.~265]{MR03}, the invariant trace field of~$\Delta_{(3,3,4)}$
(see Definition~\ref{itf}) is
\begin{equation}
\label{k delta1} k\Delta_{(3,3,4)}=\bb Q(\sqrt{2}).
\end{equation}
%% Remark: $\lambda(2,3,8) = \sqrt{2}{{-1}}$, so the invariant trace field of $\Delta_{(2,3,8)}$ is $\Q(\sqrt{2})$, same as for $\Delta_{(3,3,4)}$.
%
By Takeuchi's theorem (\cite{Take}; see \cite[Theorem 8.3.11]{MR03}),
the~$(\ell,m,n)$ triangle group is arithmetic if and only if for every
non-trivial embedding ~$\sigma$ of its invariant trace field in~$\bb
R$, we have~$\sigma\left(\lambda(\ell,m,n)\right)<0$.  The field~$\bb
Q(\sqrt{2})$ has two imbeddings in~$\bb R$.  The non-trivial imbedding
sends~$\sqrt{2}$ to~$-\sqrt{2}<0$.  Therefore by Takeuchi's theorem,
the group~$\Delta_{(3,3,4)}$ is arithmetic.

\section{Partition of Bolza surface}
\label{eight}

The
% principal congruence subgroup $\qb^1(\sqrt{2})$ contains the Fuchsian group of the
Bolza surface~$M$ %(see Section~\ref{ten}), namely
is a Riemann surface of genus 2 with a holomorphic automorphism group
of order 48, the highest for this genus.  The surface~$M$ can be
viewed as the smooth completion of its affine form
\begin{equation}
\label{81}
y^2=x^5-x
\end{equation}
where~$(x,y)\in\mathbb{C}^2$.  Here~$M$ is as a double cover of the
Riemann sphere ramified over the vertices of the regular inscribed
octahedron; this is immediate from the presentation~\eqref{81} where
the branch points are~$0, \pm 1, \pm i, \infty$.  These six vertices
lift to the Weierstrass points of~$M$.  The hyperelliptic involution
of~$M$ fixes the six Weierstass points.  It also switches the two
sheets of the cover and is a lift of the identity map on the Riemann
sphere.  The hyperelliptic involution can be thought of in affine
coordinates~\eqref{81} as the map~$(x,y)\mapsto (x,-y)$.  The
projection of~$M$ to the Riemann sphere is induced by the projection
to the~$x$-coordinate.

The surface~$M$ admits a partition into (2,3,8) triangles, which is obtained as follows.  We start with the (octahedral)
partition of the sphere into 8 equilateral hyperbolic triangles with
angle~$\pi/4$.  We then consider the barycentric subdivision, so that
each equilateral triangle is subdivided into 6 triangles of type
(2,3,8).

Here the Weierstrass points correspond to the vertices of the (2,3,8)
triangle with angle~$\pi/8$.  The partition of the Riemann sphere into
copies of the (2,3,8) triangle induces a partition of~$M$ into such
triangles.  On the sphere, we have 8 triangles meeting at each branch
point (corresponding to a Weierstrass point on the surface), for a
total angle of~$\pi$ around the branch point.  This conical
singularity is ``smoothed out'' when we pass to the double cover to
obtain the hyperbolic metric on~$M$.

To form the~$(3,3,4)$ partition, we pair up the~$\pi/8$ angles, by
combining the (2,3,8) triangles into pairs whose common side lies on
an edge of the octahedron.  This creates a partition of the sphere
into copies of the~$(3,3,4)$ triangle and induces a partition of~$M$
into copies of the~$(3,3,4)$ triangle.  Therefore the vertex of the
$(3,3,4)$ triangle where the angle is~$\pi/4$ lifts to a Weierstrass
point on~$M$.

\section{The quaternion algebra}\label{sec:4}

For the benefit of geometers who may not be familiar with quaternian
algebras, we will give a presentation following Maclachlan and Reid
2003 \cite[p.~265]{MR03} but in more detail.  To study the~$(3,3,4)$
case, we will exploit the quaternion algebra
\begin{equation}
\label{D_MR}
D_B =K\left[i,j\mid i^2=-3,\, j^2=\sqrt{2},\, ij=-ji\right]
\end{equation}
where $K=\Q(\sqrt2)$.  Denote by~$\sigma_0$ the natural embedding
of~$K$ in~$\bb R$ and by~$\sigma$ the other embedding,
sending~$\sqrt{2}$ to~$-\sqrt{2}$.

\begin{definition}
A quaternion algebra~$D$ is said to \emph{split} under a completion
(archimedean or nonarchimedean) if it becomes a matrix algebra.  It is
said to be \emph{ramified} if it remains a division algebra.
\end{definition}

\begin{remark}
\label{42}
In general there is a finite even number of places where a quaternion algebra ramifies, including the archimedean ramified places.%
\footnote{Recall that in the Hurwitz case there are two archimedean
ramified places and no nonarchimedean ones (see \cite{KSV06a}).}
Our algebra~$D_B$ ramifies at two places: the archimedean
place~$\sigma$ and the
nonarchimedean place~$(\sqrt{2})$ (see below).%
%
%
%%% The general claim follows, I think, from a discriminant argument.
%%% But it's not as easy to implement, since one must consider the prime~$\sqrt{2}$ (in the trace field, not~$\Q$),
%%% and decompose it in some maximal order.
%\footnote{Also in a situation with a quadratic field and a maximal order, apparently ramified primes correspond precisely to the rational primes which are squares of two-sided ideals in a maximal order, according the the math overflow comment at http://mathoverflow.net/questions/125043/ramified-quaternion-algebras which it would be nice to find a source for.  This does clarify a bit what is going on because in our situation indeed~$2=(\sqrt{2})^2$.}
%
\end{remark}

\begin{prop}\label{archram}
The algebra~$D_B$ splits under the natural embedding of the center in $\R$ and remains a
division algebra under the other embedding.
\end{prop}
\begin{proof}
Since~$\sqrt{2}>0$, we have
\[D_B\otimes_{\sigma_0} \bb R\cong \M_2(\bb R)\]
by \cite[Theorem 5.2.1]{Ka92}. Meanwhile, under~$\sigma$ the
algebra~$D_B$ remains a division algebra since~$-\sqrt{2}<0$, and
following \cite[Theorem 5.2.3]{Ka92}, we have~$D_B \otimes_{\sigma}\bb
R\cong\bb H$ where~$\bb H$ is the Hamilton quaternion algebra.
\end{proof}

\begin{cor}
The algebra~$D_B$ is a division algebra.
\end{cor}
\begin{proof}
Indeed~$D_B$ is a domain as a subring of~$D_B \otimes_{\sigma} \R$, and being algebraic over its center, it is a division algebra.
\end{proof}

\forget
Let~$x=x_0+x_1i+x_2j+x_3ij\in D$, then by definition \ref{Tr+Nrd}
\[\Nr_D(x)=x_0^2+3x_1^2-\sqrt{2}x_2^2-3\sqrt{2}x_3^2.\]
The nontrivial imbedding of the field~$K$ in~$\bb R$, denoted by
$\sigma$, sends~$\sqrt{2}$ to~$-\sqrt{2}$. Consider
$\sigma\left(\Nr_D(x)\right)$:
\[\begin{aligned}
\sigma\left(\Nr_D(x)\right)&=
\sigma\left(x_0^2+3x_1^2-\sqrt{2}x_2^2-3\sqrt{2}x_3^2\right)\\& =
\sigma^2(x_0)+3\sigma^2(x_1)+\sqrt{2}\left(\sigma^2(x_2)+3\sigma^2(x_3)\right).
\end{aligned}\]
Assume that~$\Nr_D(x)=0$, then~$\sigma\left(\Nr_D(x)\right)=0$ and
necessarily
\[
\sigma(x_0)=\sigma(x_1)=\sigma(x_2)=\sigma(x_3)=0
\]
and therefore~$x=0$.  By, \cite[theorem 5.2.3]{Ka92} the quaternion
algebra~$D$ is a division algebra.
\forgotten

\begin{proposition} \label{42a}
The algebra~$D_B$ ramifies at the prime~$(\sqrt{2})$ and is split at all other non-archimedean places.
\end{proposition}
\begin{proof}
The ring of integers of~$\Q(\sqrt{2})$ is~$\Z[\sqrt{2}]$,
in which the ideals~$(\sqrt{2})$ and~$(3)$ are prime.  The
discriminant of~$D_B$ is~$-6\sqrt{2}$, thus the algebra splits over any
prime other than~$(\sqrt{2})$ and~$(3)$.

Recall that $\Q_p$ denotes the field of~$p$-adic numbers, where $p$ is a rational prime. Notice that~$2$ is not a square in~$\Z/3\Z$, and therefore it cannot be a square in~$\Q_3$,
so the completion $\Q_3(\sqrt{2})$ of $\Q(\sqrt{2})$ at the prime $3$ is a quadratic extension of~$\Q_3$. To show that $D_B$ splits at $(3)$, it suffices to present~$\sqrt{2}$ as a norm in the quadratic extension~$\Q_3(\sqrt{2},\sqrt{-3})/\Q_3(\sqrt{2})$, namely in the form~$x^2+3y^2$ for~$x,y \in \Q_3(\sqrt{2})$.
%It clearly suffices to find such~$x,y \in \Z_3[\sqrt{2}]$.
By Hasse's principle, it suffices to solve the equation in the residue field~$\Z_3[\sqrt{2}]/3\Z_3[\sqrt{2}] = \F_9$, where one can take~$x = 1-\sqrt{2}$ and~$y = 0$ (indeed~$(1-\sqrt{2})^2 = 3-2\sqrt{2} \equiv \sqrt{2} \pmod{3}$).

Finally we show that~$D_B$ remains a division algebra under the completion of~$\Q(\sqrt{2})$ at the prime~$(\sqrt{2})$, which is~$\Q_2(\sqrt{2})$. It suffices to show that~$\sqrt{2}$ is not of the form~$x^2+3y^2$ for~$x,y \in \Q_2(\sqrt{2})$. Clearing out common denominators, we will show that there is no non-zero solution to~$$x^2+3y^2 = \sqrt{2}z^2$$ with~$x,y,z \in \Z_2[\sqrt{2}]$. We may assume not all of~$x,y,z$ are divisible by~$\sqrt{2}$. This equation does have a solution modulo~$4$ (indeed, take~$x = y = 1$ and~$z = 0$). We will show that there is no solution modulo~$4\sqrt{2}$. So assume
$$x^2 + 3 y^2 \equiv \sqrt{2} z^2 \pmod{4\sqrt{2}}.$$
Observe that if one of $x, y$ is divisible by $\sqrt{2}$, then they both are.  But in that case~$z$ is also divisible by~$\sqrt{2}$, contrary to assumption. So we can write~$x = 1+\sqrt{2}x'$ and~$y = 1+\sqrt{2}y'$ for~$x',y' \in \Z_2[\sqrt{2}]$. Substituting, we have
$$2\sqrt{2}+2 x' + \sqrt{2}x'^2 + 2 y'+3\sqrt{2}y'^2 \equiv z^2 \pmod{4},$$
so~$z$ is divisible by~$\sqrt{2}$ and we can write~$z = \sqrt{2}z'$ for~$z' \in \Z_2[\sqrt{2}]$. Now
$$2+\sqrt{2} x' + x'^2 + \sqrt{2} y'+3 y'^2 \equiv \sqrt{2} z'^2 \pmod{2\sqrt{2}},$$
so~$y' \equiv x' \pmod{\sqrt{2}}$, and we write~$y' = x' + \sqrt{2}y''$ for~$y'' \in \Z_2[\sqrt{2}]$. Substituting we get
$$2+ 2y'' +2y''^2 \equiv \sqrt{2} z'^2 \pmod{2\sqrt{2}},$$
so clearly~$z$ is divisible by~$\sqrt{2}$, and then
$$2+ 2y'' +2y''^2 \equiv 0 \pmod{2\sqrt{2}},$$
which implies
$$1+ y'' +y''^2 \equiv 0 \pmod{\sqrt{2}},$$
a contradiction since~$y''+y''^2$ is always divisible by~$\sqrt{2}$.
\end{proof}

\section
{The standard order in~$D_B$ and maximal orders containing
it}\label{sec:5}

In this section we prove~\Lref{intro1}.  Recall that an order~$M$
in a quaternion algebra~$D$ over a number field is maximal if and only if its discriminant is equal to the
discriminant of~$D$ \cite[Corollaire III.5.3]{Vig80}, where the
discriminant of~$D$ is the product of the ramified non-archimedean
primes. If~$M$ happens to be free as an~$O_K$-module, spanned
by~$x_1,\dots,x_4$, then its discriminant is easily computed as the
square root of the determinant of the matrix of reduced traces~$(\Tr_{D}(x_ix_j))$.

Since~$a=-3$ and~$b=\sqrt{2}$ are in~$O_K = \Z[\sqrt{2}]$, we obtain
an order~$\mcal O\subset D_B$ by setting
\[
\mcal O =O_{K}\left[i,j\right] = O_K1 + O_K i + O_K j + O_Kij .
\]
This is the ``standard order'' resulting from the presentation of~$D_B$,
for which we have~$\disc(\O)^2 = 16a^2b^2$, so that~$\disc(\O) =
12\sqrt{2}$. On the other hand~$\disc(D_B) = \sqrt{2}$ by
\Pref{archram}, so~$\O$ is not maximal. We seek a maximal order~$\QQ$
containing~$\O$. Comparing the discriminants, we know in advance
that~$\dimcol{\QQ}{\O} = 144$.

Notice that
\begin{equation}
\label{a_re}
\alpha = \frac{1}{2}(1+i)
\end{equation}
is an algebraic integer. We make the following observation.
\begin{prop}\label{p5.1}
The order~$\O_1$ generated over~$\O$ by~$\alpha$ is~$O_K[\alpha,j]$, which is spanned as a (free)~$O_K$-module by the elements
$$1,\,\alpha,\,j,\,\alpha j.$$
In particular~$\disc(\O_1) = 3\sqrt{2}$.
\end{prop}
\begin{proof}
Since~$i = 2\alpha - 1$, clearly~$\O[\alpha] = O_K[i,j,\alpha] = O_K[\alpha,j]$. To show that this module is equal to~$O_K+O_K \alpha + O_K j + O_K \alpha j$, it suffices to note that~$j^2 = \sqrt{2}$,
$$\alpha^2=\alpha{{-1}}$$
and
$$j \alpha = j - \alpha j.$$
The claim on the discriminant of~$\O_1$  then follows from computing the determinant of the $4\times 4$ traces matrix, using~$\tr(\alpha) = 1$ and~$\tr(j \alpha j) = \sqrt{2}$.
\end{proof}

Now let
\begin{equation}\label{gammadef}
\gamma = \frac{1}{6}(3+i)\left[1  - (1+\sqrt{2})j\right]
\end{equation}
and consider the~$O_K$-module
$$\QQ = O_K + O_K \alpha + O_K \gamma + O_K \alpha \gamma.$$
\begin{prop}\label{444}
The module~$\QQ$ is a maximal order of~$D_B$.  Moreover, $\QQ$
contains~$\O_1$.
\end{prop}
\begin{proof}
First note that
$$j =  (1-\sqrt{2}) (- 1  + 2 \gamma - \alpha \gamma),$$
so that~$\O \sub \O_1 \sub \QQ$.

To prove that~$\QQ$ is an order it suffices to show it is closed under multiplication, which follows by verifying the relations:
\begin{eqnarray*}
\alpha^2 & = & {{-1}} + \alpha \\
 \gamma^2 & = & (1+\sqrt{2})  +  \gamma \\
\gamma  \alpha & = & - 1 +  \alpha +  \gamma - \alpha \gamma.
\end{eqnarray*}
Maximality of~$\QQ$ follows by computation of the discriminant, which turns out to be~$\sqrt{2}$.
\end{proof}

Also let~$\gamma' = i \gamma i^{{{-1}}} = \frac{1}{6}(3+i)\left[1  + (1+\sqrt{2})j\right]$, and
$$\QQ' = O_K + O_K \alpha + O_K \gamma' + O_K \alpha \gamma'.$$

Notice that~$\QQ' = i \QQ i^{{{-1}}}$  is conjugate to~$\QQ$.
\begin{cor}
The module~$\QQ'$ is a maximal order containing~$\O_1$.
\end{cor}
\begin{proof}
This is immediate because~$i \O_1 i^{{{-1}}} = \O_1$.
\end{proof}
\forget %DDD
\begin{prop}
A direct proof that
$$\QQ^{\theta} = O_K + O_K \alpha + O_K \alpha j + O_K \gamma_{\theta}$$
is an order.
\end{prop}
\begin{proof}
\begin{eqnarray*}
\alpha \theta \gamma_{\theta} & = & \theta\alpha - (1+\sqrt{2}) \alpha j - \theta \gamma_{\theta} \\
\alpha j \theta \gamma_{\theta} & = & {{-1}}- \alpha  +  \theta (1-\sqrt{2})\theta \gamma_{\theta} \\
\theta \gamma_{\theta} \alpha & = & - \theta  + (1+\sqrt{2}) \alpha j + 2 \theta \gamma_{\theta} \\
\theta \gamma_{\theta} \alpha j & = &  - (1+\sqrt{2}) + \alpha + \theta \alpha j - \theta (1-\sqrt{2})\theta \gamma_{\theta} \\
 \gamma_{\theta}^2 & = & \gamma_{\theta}  + (1+\sqrt{2}).
\end{eqnarray*}
\end{proof}
\forgotten

\begin{prop}
The only two maximal orders containing~$\O$ are~$\QQ$ and~$\QQ'$.
\end{prop}
\begin{proof}
Let~$y \in D_B$
be an element such that~$\O[y]$ is an order. Write
$$y = \frac{1}{2}(x_0+\frac{x_1}{3}i+\frac{x_2}{\sqrt{2}}j+\frac{x_3}{3\sqrt{2}}ij),$$
where~$x_0,x_1,x_2,x_3 \in \Q(\sqrt{2})$. Since~$\tr(y \O) \sub O_K$, we immediately conclude that in fact~$x_0,x_1,x_2,x_3 \in \Z[\sqrt{2}]$. Furthermore, the norm of~$y$ is an algebraic integer, proving that~$12\sqrt{2}$ divides
$$-3\sqrt{2}x_0^2-\sqrt{2}x_1^2+3x_2^2+x_3^2$$
in~$\Z[\sqrt{2}]$. Working modulo powers of~$\sqrt{2}$, we conclude as in \Pref{archram} that
$x_3 = x_2+2\sqrt{2}x_3'$,~$x_1 = x_0+2x_1'$,~$x_2 = \sqrt{2} x_2'$ for suitable~$x_1',x_2',x_3' \in \Z[\sqrt{2}]$. The remaining condition is that~$(x_0-x_1')^2 \equiv \sqrt{2}(x_2'-x_3')^2 \pmod{3}$, so in fact
$$x_0 = x_1' + \theta(1-\sqrt{2})(x_2'-x_3') + 3x_0'$$
for some~$x_0' \in \Z[\sqrt{2}]$ where~$\theta = \pm 1$. But then%
\begin{eqnarray*}
y-x_0' & = & \frac{1}{2}(1+i)(x_0'+x_1') + \frac{1}{2}(j+ij)x_3'  \\
    & & \qquad  + \frac{1}{6}\left[\theta(1-\sqrt{2})(3+i)+3j+ij\right](x_2'-x_3') \\
& = & (x_0'+x_1')\alpha + x_3' \alpha j + (x_2'-x_3')(1-\sqrt{2})\theta \gamma_{\theta},
\end{eqnarray*}
where~$\gamma_{+1} = \gamma$ and~$\gamma_{{{-1}}} = \gamma'$. Thus~$y$ is an element of~$\QQ$ (if~$\theta = 1$) or of~$\QQ'$ (if~$\theta = {{-1}}$).
\end{proof}

Note that~$\QQ + \QQ'$ is not an order, since~$\gamma + \gamma' = 1+\frac{i}{3}$ is not an algebraic integer.

\section{The Bolza order}\label{sec:6}

In order to present the triangle group~$\Delta_{(3,3,4)}$ as a quotient of the group of units in a maximal order, we make the following change of variables.
Let
\begin{equation}\label{b_re}
\beta =\frac{1}{6}\left(3+(1+2\sqrt{2})i- 2 ij\right).
\end{equation}

Since
$$\beta  =  \alpha(1- (1-\sqrt{2})\gamma)$$
(where $\gamma$ is defined in \eq{gammadef}) and
$$\gamma  =  - (1+\sqrt{2})(1- \beta  + \alpha\beta),$$
we have that
$$\qb := O_K[\alpha,\beta] = \QQ.$$
In particular,~$\qb$ is a maximal order by \Pref{444}.

One has
\begin{equation}
\label{45b}
\alpha\beta=-\frac{1}{6}\left(3\sqrt{2}-(2+\sqrt{2})i+3j-ij\right).
\end{equation}

\begin{theorem}\label{43b}
The order~$\qb$ is spanned as a module over~$O_K$ by the basis
$\{1,\alpha,\beta,\alpha\beta\},$ so that
\begin{equation}
\label{45}
\qb=O_K1\oplus O_K\alpha\oplus O_K\beta\oplus O_K\alpha\beta.
\end{equation}
\end{theorem}
\begin{proof}
Let~$M=O_K1 + O_K\alpha+ O_K\beta+ O_K\alpha\beta$.  The following
relations are verified by computation:
\begin{enumerate}
\item~$\alpha^2={{-1}}+\alpha$, % verified.
\item~$\beta^2={{-1}}+\beta$, % verified.
\item~$\beta\alpha=({{-1}}-\sqrt{2})+\alpha+\beta-\alpha\beta$; % verified.
\end{enumerate}
and thus~$\alpha(\alpha\beta) = -\beta + \alpha\beta \in M$ and
\[
\beta(\alpha\beta) =
({{-1}}-\sqrt{2})\beta+\alpha\beta+\beta^2-\alpha\beta^2 = {{-1}} +\alpha
-\sqrt{2}\beta \in M.
\]
It follows that~$\alpha M, \beta M \sub M$, so~$M$ is
closed under multiplication and is therefore equal to~$\qb$.

\forget
%% Details should be omitted.
%
To show that the sum defining~$\qb'$ is direct, notice that
\begin{eqnarray*}
& &x_0 + x_1 \alpha + x_2 \beta + x_3 \alpha \beta = \\
& &\qquad (x_0 + \frac{1}{2}x_1 + \frac{1}{2} x_2 - \frac{\sqrt{2}}{2} x_3 ) \\
&& \qquad +(\frac{1}{2}x_1  + \frac{1+2\sqrt{2}}{6}x_2 + \frac{2+\sqrt{2}}{6} x_3 )i \\
&& \qquad - \frac{1}{2} x_3 j
+(\frac{1}{3}x_2 + \frac{1}{6}x_3) ij,\end{eqnarray*}
so if this sum is zero, we have that~$x_3 = 0$ from the coefficient of~$j$,~$x_2 = 0$ from the coefficient of~$ij$, and then~$x_1 =0$ from the coefficient of~$i$ and~$x_0 = 0$ from the free coefficient.
\forgotten
\end{proof}

\forget
\begin{proposition}
\label{44}
One has the following relations:
\begin{enumerate}
\item
$\alpha^2={{-1}}+\alpha$ % verified.
\item
$\beta^2={{-1}}+\beta$ % verified.
\item
$ \gamma_{\theta}^2={{-1}}-\sqrt{2}\theta \gamma_{\theta}$
\item~$\beta\alpha=({{-1}}-\sqrt{2})+\alpha+\beta-\alpha\beta$ % verified.
 \item~$\alpha\theta \gamma_{\theta}=-\beta+\theta \gamma_{\theta}$
 \item~$\theta \gamma_{\theta}\alpha={{-1}}-\sqrt{2}\alpha+\beta$
 \item~$\beta\theta \gamma_{\theta}={{-1}}+\alpha-\sqrt{2}\beta$
 \item~$\theta \gamma_{\theta}\beta=-\alpha+\theta \gamma_{\theta}$.
\end{enumerate}
\end{proposition}

\begin{proof}[Proof of Proposition~\ref{44}]
We first check formula~\eqref{45b}.  We have
\[
\begin{aligned}
\alpha\beta&= \frac{1}{12}(1+i)\left(3+(1+2\sqrt{2})i+2ij\right)
\\&= \frac{1}{12}\left(3+(1+2\sqrt{2})i+2ij+3i+(1+2\sqrt{2})ii+2iij\right)
\\&= \frac{1}{12}\left(3+(1+2\sqrt{2})i+2ij+3i-3(1+2\sqrt{2})-6j\right)
\\&= \frac{1}{12}\left(3+(4+2\sqrt{2})i+2ij-3-6\sqrt{2}-6j\right)
\\&=\frac{1}{6}\left((2+\sqrt{2})i+ij-3\sqrt{2}-3j\right),
\end{aligned}
\]
as required.  Next,
$\alpha^2=\frac{1}{4}(1+i)^2=\frac{1}{4}(1-3+2i)=\frac{1}{2}(i{{-1}})$ and
${{-1}}+\alpha=\frac{1}{2}(-2+1 + i)=\frac{1}{2}(i{{-1}})$, establishing
relation~(1).  Let~$r=\sqrt{2}$.  One has
\[
\begin{aligned}
\beta^2 &= \frac{1}{36}\left(3+(1+2r)i+2ij\right)^2
\\&=
\frac{1}{36}\left(9+(1+2r)^2i^2+4ijij+6(1+2r)i+12ij+2(1+2r)(iij+iji)\right)
\\&=
\frac{1}{36}\left(9-3(9+4r)+12jj+6(1+2r)i+12ij\right)
\\&=
\frac{1}{36}\left({{-1}}8{{-1}}2r+12r+6(1+2r)i+12ij\right)
\\&= \frac{1}{6}\left(-3+(1+2r)i+2ij\right) \\&={{-1}}+\beta,
\end{aligned}
\]
establishing the relation~(2).  To check~(3), recall that
$r=\sqrt{2}$.  We have
\[
\begin{aligned}
 \gamma_{\theta}^2 &= (\alpha\beta)^2 \\&=
\frac1{36}\left(3\sqrt{2}-(2+\sqrt{2})i+3j-ij\right)^2 \\&=
\frac1{36}\left(18-3(2+r)^2+9r+ijij-6r(2+r)i+18rj-6rij\right.  \\&
\left.\quad\quad\quad
{}^{\phantom{2}}-3(2+r)(ij+ji)+(2+r)(iij+iji)-3(jij+ijj)\right) \\&=
\frac1{36}\left(18-3(2+r)^2+9r+ijij-6r(2+r)i+18rj-6rij\right)
\\&=\frac1{36}\left(18-3(6+4r)+12r-6r(2+r)i+18rj-6rij\right)
\\&=\frac1{6}\left(-(2r+2)i+3rj-rij\right).
\end{aligned}
\]
Meanwhile,
\[
\begin{aligned}
{{-1}}-\sqrt{2}\theta \gamma_{\theta}&={{-1}} +\frac{r}{6}\left(3r-(2+r)i+3j-ij\right)
\\&= {{-1}} +\frac{1}{6}\left(3r^2-(2r+r^2)i+3rj-rij\right)
\\&= {{-1}} +\frac{1}{6}\left(6-(2r+2)i+3rj-rij\right)
\\&= \frac{1}{6}\left(-(2r+2)i+3rj-rij\right)
\\&=  \gamma_{\theta}^2
\end{aligned}
\]
by the previous calculation, establishing the relation~(3).  Let us
check the relation (4).  We have
$\beta=\frac{1}{6}\left(3+(1+2r)i+2ij\right)$ and
$\alpha=\frac12\left(1+i\right)$.  Therefore the left-hand side of (4)
is
\[
\begin{aligned}
\beta\alpha &= \frac1{12}\left(3+(1+2r)i+2ij\right)\left(1+i\right)
\\&= \frac1{12}\left(3+ 3i + (1+2r)i+(1+2r)ii+2ij + 2iji\right)
\\&=\frac1{12}\left(3+(4+2r)i-3(1+2r)+2ij-2iij\right)
\\&=\frac1{12}\left((4+2r)i-6r+2ij+6j\right)
\\&=\frac1{6}\left((2+r)i-3r+ij+3j\right).
\end{aligned}
\]
Meanwhile, the right-hand side of (4) is~$\text{rhs}=
({{-1}}-r)+\alpha+\beta-\theta \gamma_{\theta}$ or more explicitly
\[
\begin{aligned}
\text{rhs}&=
\frac{1}{6}\left(6({{-1}}-r)+3\left(1+i\right)+3+(1+2r)i+2ij+3r-(2+r)i+3j-ij\right)
\\&= \frac{1}{6}\left(-6-6r+3+3i+3+(1+2r)i+2ij+3r-(2+r)i+3j-ij\right)
\\&= \frac{1}{6}\left(-3r+3i+(1+2r)i+ij-(2+r)i+3j\right) \\&=
\frac{1}{6}\left(-3r+(2+r)i+ij+3j\right) \\&= \beta\alpha
\end{aligned}
\]
by the previous calculation, establishing the relation~(4).  To check
the relation~(5) asserting that~$\alpha\theta \gamma_{\theta}=-\beta+\theta \gamma_{\theta}$, note
that by definition~$\alpha\theta \gamma_{\theta}=\alpha\alpha\beta= ({{-1}}+\alpha)\beta$
by relation (1), and therefore~$\alpha\theta \gamma_{\theta}=-\beta+\theta \gamma_{\theta}$ by
definition of~$\theta \gamma_{\theta}$.

To check relation~(6) to the effect that
$\theta \gamma_{\theta}\alpha={{-1}}-\sqrt{2}\alpha+\beta$, note that
$\theta \gamma_{\theta}\alpha=\alpha\beta\alpha=
\alpha\left(({{-1}}-r)+\alpha+\beta-\theta \gamma_{\theta}\right)$ by relation~(4).
Applying the previously derived relations, we obtain
\[
\begin{aligned}
\theta \gamma_{\theta}\alpha &= -\alpha-r\alpha+\alpha^2+\alpha\beta-\alpha\theta \gamma_{\theta}
\\&=-\alpha-r\alpha+({{-1}}+\alpha)+\alpha\beta-\alpha\theta \gamma_{\theta}
\\&=-r\alpha{{-1}}+\alpha\beta-\alpha\theta \gamma_{\theta}
\\&=-r\alpha{{-1}}+\theta \gamma_{\theta}-\alpha\theta \gamma_{\theta}
\\&=-r\alpha{{-1}}+\theta \gamma_{\theta}-(-\beta+\theta \gamma_{\theta})
\\&=-r\alpha{{-1}}+\beta,
\end{aligned}
\]
establishing the relation~(6).

We will check (8) before (7).  To establish the relation~(8) to the
effect that~$\theta \gamma_{\theta}\beta=-\alpha+\theta \gamma_{\theta}$, note that
$\theta \gamma_{\theta}\beta=\alpha\beta^2$, and therefore by relation~(2), we have
$\theta \gamma_{\theta}\beta=\alpha({{-1}}+\beta)=-\alpha+\alpha\beta=-\alpha+\theta \gamma_{\theta}$ by
definition of~$\theta \gamma_{\theta}$, establishing (8).  Finally, to establish the
relation~(7) to the effect that~$\beta\theta \gamma_{\theta}={{-1}}+\alpha-r\beta$, note
that by relation~(4), we have~$\beta\theta \gamma_{\theta}=\beta\alpha\beta=
\left(({{-1}}-r)+\alpha+\beta-\theta \gamma_{\theta}\right)\beta$.  Therefore
\[
\begin{aligned}
\beta\theta \gamma_{\theta}&= -\beta-r\beta+\alpha\beta+\beta^2-\theta \gamma_{\theta}\beta
\\&=-\beta-r\beta+\theta \gamma_{\theta}+\beta^2-\theta \gamma_{\theta}\beta
\\&=-\beta-r\beta+\theta \gamma_{\theta}+\beta^2+\alpha-\theta \gamma_{\theta}
\\&=-\beta-r\beta+\beta^2+\alpha \\&=-\beta-r\beta{{-1}}+\beta+\alpha
\\&=-r\beta{{-1}}+\alpha,
\end{aligned}
\]
establishing the relation~(7).
\end{proof}
\forgotten

\forget
\begin{proposition}
We have
\begin{equation}
\label{O inclusion} \mcal O\subseteq\qb.
\end{equation}
\end{proposition}
\begin{proof}
We have that~$i = {{-1}} + 2\alpha$ and~$j = -(1+\sqrt{2})+\alpha+\beta-2\alpha\beta$.

%\forget
To prove the inclusion \eqref{O inclusion}, let
$X=m_0+m_1i+m_2j+m_3ij\in\mcal O$ (where~$m_\ell\in O_K$).  Then
\[
\begin{aligned}
X & =\left((m_0-m_1-m_2-m_3)+(m_3-m_2)\sqrt{2}\right)
\\&\quad+\left((2m_1+m_2-m_3)-2m_3\sqrt{2}\right)\alpha
\\&\quad+\left(m_2+3m_3\right)\beta
\\&\quad-\left(2m_2\right)\alpha\beta,
\end{aligned}
\]
and therefore~$X\in\qb$.
%\forgotten
\end{proof}
\forgotten

\forget
\begin{remark}
The quotient~$\qb/{\mcal O}$ is an abelian group of order~$144$ and exponent~$6$, isomorphic as an~$O_K$-module
to~$(O_K/2O_K) \oplus (O_K/2O_K) \oplus (O_K/3O_K)$.
\end{remark}

\begin{proof}
We view~$\mcal O$ as a submodule of~$\qb$. Presenting the elements~$1,i,j$ and
$ij  = ({{-1}}+\sqrt{2}) - (1+2\sqrt{2})\alpha  + 3\beta$, which span~$\mcal O$, as linear combinations of the basis~$1,\alpha,\beta,\alpha\beta$, we obtain the relation matrix of the quotient:
$$\left(\begin{array}{cccc}
1 & {{-1}} & {{-1}}+\sqrt{2} & {{-1}}+\sqrt{2} \\
0 & 2  & 1         & -(1+2\sqrt{2}) \\
0 & 0  & 1         & 3 \\
0 & 0  & -2        & 0
\end{array}\right).$$
Performing elementary operations from left and right, over~$O_K$, we
arrive at the matrix~$\operatorname{diag}(1,1,2,6)$, which proves that
\[
\qb/{\mcal O} \,\cong\, (O_K/2O_K) \oplus (O_K/6O_K),
\]
from which the claim follows by the Chinese remainder theorem.
\end{proof}
\forgotten

\forget
\begin{proposition}
\label{56}
The order~$\qb$ is a maximal order of the quaternion algebra~$D_B$.
\end{proposition}
\begin{proof}
By \Tref{43b}, a basis of~$\qb$ over~$O_K$ is given by the set~$\{ b_1, b_2, b_3, b_4 \}$, where~$(b_1, b_2, b_3, b_4) = (1, \alpha, \beta, \alpha \beta)$.  It follows from the relations (1)-(3) of \Tref{43b}, as well as the relations deduced from them in that proof, that the matrix~$( \Tr_{D_B} (b_i b_j)) \in \M_4(K)$ is:
$$
\left(
\begin{array}{cccc}
2 & 1 & 1 & - \sqrt{2} \\
1 & {{-1}} & - \sqrt{2} & {{-1}} - \sqrt{2} \\
1 & - \sqrt{2} & {{-1}} & {{-1}} - \sqrt{2} \\
- \sqrt{2} & {{-1}} - \sqrt{2} & {{-1}} - \sqrt{2} & 0
\end{array}
\right).
$$

This matrix has determinant~$-2$, and therefore the reduced discriminant of~$\qb$ is the prime ideal~$(\sqrt{2})$.  It now follows from Proposition~\ref{42a} and \cite{Vig80}, Corollaire III.5.3 that~$\qb$ is a maximal order.
\end{proof}
\forgotten

\section{The triangle group in the Bolza order}
\label{sec:7}

Let~$\qb^1$ denote the group of elements of norm~$1$ in the
order~$\qb$.  Through the embedding $D_B \hookrightarrow \M_2(\R)$, we
may view~$\qb^1$ as an arithmetic lattice of~$\SL(\R)$. Furthermore,
by \Pref{archram} the algebra~$D_B$ ramifies at all the archimedean
places except for the natural one, so it satisfies Eichler's
condition; see \cite[p.~82]{Vig80}.  Therefore~$\qb^1$ is a co-compact
lattice.

Since~$\Norm(\alpha) = \Norm(\beta) = 1$, the subgroup generated
by~$\alpha,\beta$ in~$\mul{D_B}$ is contained in~$\qb^1$.

\begin{prop}
The elements~$\alpha,\beta$ defined in \eqref{a_re} and \eqref{b_re}
satisfy the relations
$$\alpha^3 = \beta^3 = (\alpha\beta)^4 = {{-1}}.$$
\end{prop}
\begin{proof}
First we note that~$\Norm(\alpha) = \Norm(\beta) = 1$.
The minimal polynomial of every non-scalar element of~$D_B$ is quadratic, determined by the trace and norm of the element. Since~$\tr(\alpha) = \tr(\beta) = 1$, both~$\alpha$ and~$\beta$ are roots of the polynomial~$\lam^2 - \lam + 1$, which divides~$\lam^3 +1$. Similarly~$\tr(\alpha\beta) = - \sqrt{2}$, so~$\alpha \beta$ is a root of~$\lam^2 + \sqrt{2}\lam +1$, which divides~$\lam^4 +1$.
% Some of the relations were computed before, in \Tref{43b}.
\end{proof}

A comparison of the areas of the fundamental domains shows that in fact~$\qb^1 = \sg{\alpha,\beta}$ and that~$\qb^1/\set{\pm 1}$ is isomorphic to the triangle group~$\Delta_{(3,3,4)}$.

\section{A lower bound for the systole}\label{sec:8}

We give lower bounds on the systole of congruence covers of any
arithmetic surface and then specialize to the Bolza surface.  Let $K$
be any number field, $O_K$ its ring of integers, $D$ any central division algebra over $K$, and
$Q$ an order in $D$. Let $X_1 = \mcal{H}^2/Q^1$, where $Q^1$ is the
group of elements of norm $1$ in $Q$. We let $d = \dimcol{K}{\Q}$.

We quote the definition of the constant $\Lambda_{D,Q}$ from
\cite[Equation~(4.9)]{KSV06a}.  Let~$T_1$ denote the set of finite
places $\frak p$ of $K$ for which~$D_{\frak p}$ is a division algebra,
and let~$T_2$ denote the set of finite places for which~$Q_{\frak p}$
is non-maximal. It is well known that~$T_1$ and~$T_2$ are finite.  We
denote
\begin{equation}\label{lamdef}
\Lambda_{D,Q} = \prod_{\frak{p} \in T_1\setminus T_2}
{\left(1+\frac{1}{\Norm(\frak{p})}\right)} \cdot \prod_{\frak{p} \in T_2} 2 \cdot
\prod_{\frak{p} \in T_2, \frak{p}\divides 2} \Norm(\frak{p})^{e(\frak{p})},
\end{equation}
where for a diadic prime,~$e(\frak{p})$ denotes the ramification index
of~$2$ in the completion $O_{\frak{p}}$, namely $\frak{p}^{e(\frak{p})}O_{\frak{p}} = 2O_{\frak{p}}$, and $\Norm(I)$ denotes the norm of the ideal $I$. This constant is chosen in \cite{KSV06a} to ensure that $\dimcol{Q^1}{Q^1(I)} \leq \Lambda_{D,Q} \Norm(I)^3$, for any ideal $I$.

Recall that if~$I \normali O_K$ is any ideal, then~$Q^1(I)$ is the kernel of the natural map~$Q^1 \ra
(Q/IQ)^1$ induced by the ring epimorphism~$Q \ra Q/IQ$. This
congruence subgroup gives rise to the surface~$X_I = \mcal
H^2/Q^1(I)$, which covers~$X_1$. A bound for the reduced trace was given in \cite[Equation~(2.5)]{KSV06a} as follows. Let $x\neq\pm1$ in~$Q^1(I)$. Then we have
\begin{equation}\label{81c}
\abs{\Tr_D(x)}>\frac{1}{2^{2(d{{-1}})}} \Norm(I)^2-2. % This can actually be improved to \frac{1}{2^{d-2} \Norm(\ideal{2}+\kappa I)} \Norm(I)^2 - 2, cf. (2.5) in [KSV06a].
\end{equation}

By \cite[Corollary 4.6]{KSV06a}, we have
\[
[Q^1:Q^1(I)]\leq\Lambda_{D,Q}N(I)^3.
\]
Therefore
\[\begin{aligned} 4\pi\left(g(X_I){{-1}}\right)&\leq\area(X_I)
 \\
 &=\dimcol{Q^1}{Q^1(I)}\cdot\area(X_1) \\
& \leq\Lambda_{D,Q}N(I)^3\cdot\area(X_1),\end{aligned}\]
i.e.
\[N(I)\geq\left(\frac{4\pi}{\Lambda_{D,Q}\cdot\area(X_1)}(g{{-1}})\right)^{\frac13}.\]

\begin{proposition}
\label{81b}
Suppose $2^{3(d{{-1}})}\Lambda_{D,Q} < \frac{4\pi}{\area(X_1)}$.
Then all but finitely many principal congruence covers of $X_1$
satisfy the relation $$\sys > \frac{4}{3}\log g.$$
\end{proposition}

\begin{proof}
A hyperbolic element~$x$ in a Fuchsian
group~$\Gamma\subseteq\PSL(\bb R)$ is conjugate to a matrix
\[
\left( \begin{array}{cc} \lambda & 0 \\ 0 & \lambda^{{{-1}}} \\
\end{array} \right).
\]

Here~$\lambda=e^{\ell_x/2}>1$, where~$\ell_x>0$ is the length of the
closed geodesic corresponding to~$x$ on the Riemann surface~$\mcal
H^2/\Gamma$.  Since
\[
\left|\Tr_{M_2(\bb R)}(x)\right|=
\left|\lambda+\lambda^{{{-1}}}\right|\leq\left|\lambda\right|+
\left|\lambda^{{{-1}}}\right|\leq\left|\lambda\right|+1,\] we get
\[\ell_x=2\log|\lambda|>2\log\left(\left|\Tr_{M_2(\bb
R)}(x)\right|{{-1}}\right).\]

By \eq{81c},
%\[\begin{aligned}
\begin{eqnarray}
\sys(X_I) & > & 2\log\left(|\Tr_D(x)|{{-1}}\right) \nonumber \\
& > & 2\log\left( \frac{1}{2^{2(d{{-1}})}} \Norm(I)^2-3\right) \label{eq82} \\
& \geq & 2\log\left(\frac{1}{2^{2(d{{-1}})}}\left[\frac{4\pi}{\Lambda_{D,Q} \cdot \area(X_1)}\left(g(X_I){{-1}}\right)\right]^{\frac23}-3\right). \nonumber
\end{eqnarray}
%\end{aligned}\]
Expanding the argument under the logarithm as a series in $g$, we find that the coefficient of the highest term $g^{2/3}$
is $\left[\frac{1}{2^{3(d{{-1}})}}\frac{4\pi}{\Lambda_{D,Q} \cdot \area(X_1)}\right]^{\frac23}$. When this coefficient is strictly greater than $1$, for sufficiently large $g$ we have that
\[
\begin{aligned}
\sys(X_I) &
> \frac{4}{3} \log \left(g(X_I)\right).
\end{aligned} \qedhere
\]
\end{proof}

A closer inspection of \eqref{eq82} enables us to provide an explicit bound on the genera $g$ for which the inequality of \Pref{81b} holds.
\begin{remark}\label{81explicit}
We have that
$$2\log\left(\frac{1}{2^{2(d{{-1}})}}\left[\frac{4\pi}{\Lambda_{D,Q} \cdot \area(X_1)}\left(g{{-1}}\right)\right]^{\frac23}-3\right) > \frac{4}{3}\log(g)$$
if and only if
$$\frac{\left(1+\frac{3}{g^{2/3}}\right)^{3/2}}{1-\frac{1}{g}} \leq \frac{4\pi}{2^{3(d{{-1}})}\Lambda_{D,Q} \cdot \area(X_1)}.$$
Since
$$\frac{\left(1+\frac{3}{g^{2/3}}\right)^{3/2}}{1-\frac{1}{g}} \leq 1+ \frac{6}{g^{2/3}}$$
for every $g \geq 13$, we conclude that if $2^{3(d{{-1}})}\Lambda_{D,Q} < \frac{4\pi}{\area(X_1)}$, then
$\sys > \frac{4}{3}\log g$ provided that
$$g \geq \max\set{13, \left(\frac{6}{\frac{4\pi}{2^{3(d{{-1}})}\Lambda_{D,Q} \area(X_1)}{{-1}}}\right)^{3/2}}.$$
\end{remark}

\forget
\begin{remark}
Using the tighter bound of Equation~(2.5) in \cite{KSV06a} rather than (2.6), we can make a stronger statement. Suppose the quaternion algebra $D$ is presented in the form of \eq{quaternions}, and that $Q \sub \frac{1}{\kappa}O_K[i,j]$ for $\kappa \in O_K$ (such $\kappa$ always exists).

The assumption in \Pref{81b} can be weakened to
$(2^{d-2}\Norm(\ideal{2}+\kappa I))^{3/2}\Lambda_{D,Q}  < \frac{4\pi}{\area(X_1)}$. This is an improvement only for odd ideals $I$, and where $\kappa$ is odd.
\end{remark}
\forgotten

\begin{corollary}
Principal congruence covers of the Bolza order satisfy the bound
$\sys > \frac{4}{3}\log g$ provided that $g \geq 15$.
\end{corollary}

\begin{proof}
Since the order~$\qb$ is maximal, it follows (e.g. by \cite[Corollary 6.2.8]{MR03}) that all localisations are maximal as well.
Therefore the set~$T_2$ is empty (see material around
\cite[formula~4.10]{KSV06a}), while~$T_1$ consists of a single
nonarchimedean place~$\sqrt{2}$ with norm~$2$ (see Remark~\ref{42}).
Therefore~$\Lambda_{D_B,\qb}=\frac{3}{2}$.

Moreover, since $\qb^1/\{\pm 1\}$ is the triangle group $(3,3,4)$, we
have
\[
\area(X_1) = 2
\left(\pi-\left(\frac\pi3+\frac\pi3+\frac\pi4\right)\right)=
\frac{\pi}{6},
\]
so $\frac{4\pi}{\area(X_1)} = 24$.  Finally the dimension of the
invariant trace field over~$\Q$ is $d = 2$, so the condition
$2^{3(d{{-1}})}\Lambda_{D_B,\qb} < \frac{4\pi}{\area(X_1)}$ of \Pref{81b}
holds since $12 < 24$.

In order to obtain the explicit lower bound on $g$, we substitute in
\Rref{81explicit}, using the numerical value $6^{3/2} \approx 14.697$.
\end{proof}

\section{The Fuchsian group of the Bolza surface}
\label{ten}

In this section we give an explicit presentation of the Fuchsian group
of the Bolza surface in terms of the quaternion
algebra~$\mathcal{Q}_B$.  We start with a geometric lemma that will
motivate the introduction of the special element exploited in
Lemma~\ref{112}.

\begin{lemma}
\label{111}
Let~$\bar A$ and~$\bar B$ be antipodal points on a systolic loop of a
hyperbolic surface~$M$.  Let~$A$ and~$B$ be their lifts to the
universal cover such that~$d(A,B)=\frac{1}{2}\sys(M)$.  Let~$\tau_A$
and~$\tau_B$ be the involutions of the universal cover with centers
at~$A$ and~$B$.  Then the composition~$\tau_B\circ\tau_A$ belongs to a
%
%''a'', not ``the''
%
conjugacy class in the fundamental group defined by the systolic loop.
\end{lemma}

\begin{proof}
A composition of two involutions gives a translation by twice the
distance between the fixed points of the involutions.  Thus, consider
the hyperbolic line $\rho$ in the universal cover passing through $A$
and $B$.  Then the composition $\tau_B\circ\tau_A$ is a hyperbolic
translation along $\rho$ with displacement distance precisely
$\sys(M)$.  The image of the projection of $\rho$ back to $M$ is the
systolic loop.
\end{proof}

We now apply Lemma~\ref{111} in a situation where the points~$A$
and~$B$ are lifts of Weierstrass points on the Bolza surface (see
Section~\ref{eight} for details).  The composition of the
involutions~$(\alpha \beta)^2$ and~$(\beta\alpha)^{2}$ yields the
desired element.  This element was obtained through a detailed
geometric analysis of the action in the upperhalf plane which we will
not reproduce (\magma{} was not used here).

\begin{lemma}
\label{112}
The element~$(\alpha \beta)^2(\beta\alpha)^{-2}$ is in the congruence
subgroup $\qb^1(\sqrt{2})$.
\end{lemma}
\begin{proof}
One has~$(\alpha\beta)^2(\beta\alpha)^{2}=
1+\sqrt{2}(1+(1+\sqrt{2})(\alpha-\beta))$.
\end{proof}

\begin{proposition}
\label{102}
The normal subgroup of the~$(3,3,4)$ triangle group generated by the
element~$(\alpha \beta)^2 (\beta\alpha)^{-2}$ has index~$24$.  The
normal subgroup is generated by the following four elements:
\begin{itemize}
\item
$c_1 = \alpha^{{{-1}}} \beta \alpha \beta^{{{-1}}} \alpha \beta$,
%
%%%UV changed the leading \alpha^2 to \alpha^{{{-1}}}
%
\item
$c_2 = \alpha \beta^{{{-1}}} \alpha \beta \alpha^{{{-1}}} \beta$,
\item
$c_3 = \alpha \beta^{{{-1}}} \alpha^{{{-1}}} \beta \alpha^{{{-1}}} \beta^{{{-1}}}$,
\item
$c_4 = \beta \alpha^{{{-1}}} \beta \alpha \beta^{{{-1}}} \alpha$,
\end{itemize}
which satisfy a single length-$8$ relation~$c_4^{{{-1}}} c_3^{{{-1}}}
c_2 c_4 c_1 c_2^{{{-1}}} c_1^{{{-1}}}c_3 = 1$.
%
%equivalently: $[e,f] = [d, c] [ c d c^{{{-1}}}, f]$ , where~$[x,y]$
%denotes the commutator~$xyx^{{{-1}}}y^{{{-1}}}$.
%
The reduced traces are
$$\tr(c_1) = \tr(c_2) = \tr(c_3) = \tr(c_4) = -2(1 + \sqrt{2}).$$
\end{proposition}

This was checked directly using the \magma{} package.

\begin{corollary}\label{114}
The normal subgroup of~$\mathcal{Q}_B^1$ generated by the element
$(\alpha \beta)^2 (\beta\alpha)^{-2}$ generates the Fuchsian group of
the Bolza surface.
\end{corollary}

\begin{proof}
The presentation of the Fuchsian group given in Proposition~\ref{102}
implies that the surface has genus 2.  This identifies it as the Bolza
surface which is the unique genus-2 surface admitting a tiling of type
$(3,3,4)$ or (2,3,8); see Bujalance \& Singerman (1985
\cite[p.~518]{BS85}).  This surface is known to have the largest
systole in genus~$2$, or equivalently largest trace~$2(1+\sqrt{2})$
(see e.g., Bavard \cite[p.~6]{Bav2}, Katz \& Sabourau \cite{KS},
Schmutz \cite{Sc1}).  Therefore all 4 generators specified in
Proposition~\ref{102} correspond to systolic loops.
\end{proof}

\section{An elliptic element of order 2}
\label{nine}

The principal congruence subgroup $\qb^1(\sqrt{2})$ contains the
Fuchsian group of the Bolza surface (see \Lref{112}), but it also
contains torsion elements.  The element
\begin{equation}\label{varpidef}
\varpi=1+\sqrt{2}\alpha\beta
\end{equation}
in $\qb^1(\sqrt{2})$ defines an elliptic (torsion) element of
order~$2$ in the Fuchsian group.  Indeed,
%
%by Proposition~\ref{44}, item~(3)
%
applying the relations given in \Tref{43b}, we have~$(\alpha
\beta)^2={{-1}}-\sqrt{2}\alpha\beta$.  Hence
\[
\varpi^2=(1+ \sqrt{2} \alpha\beta)^2 = 1 + 2 \sqrt{2}\alpha\beta + 2
(\alpha\beta)^2 = {{-1}}
\]
and therefore~$\varpi$ is of order 2 in the Fuchsian group.

By the above,~$\varpi = - (\alpha\beta)^2$.  The fixed point of~$\varpi$
can be taken to be the vertex of a~$(3,3,4)$ triangle where the angle
is~$\pi/4$.  The element~$\alpha\beta$ gives a rotation by~$\pi/2$ around
this vertex, and therefore~$\varpi$ gives the rotation by~$\pi$ around the
vertex of the~$(3,3,4)$ triangle where the angle is~$\pi/4$.

\begin{lemma}
The action of~$\varpi$ descends to the Bolza surface and coincides with the
hyperelliptic involution of the surface.
\end{lemma}

\begin{proof}
The involution~$\varpi$ is a rotation by~$\pi$ around a Weierstrass point
(see Section~\ref{eight}), namely the vertex of the~$(3,3,4)$ triangle
where the angle is~$\pi/4$.  Therefore~$\varpi$ descends to the identity on
the Riemann sphere.  Thus~$\varpi$ lifts to the hyperelliptic involution
of~$M$.
\end{proof}

\section{Quotients of the Bolza order}\label{sec:12}

In the next section we compare the Bolza group with some principal
congruence subgroups of the Bolza order. To this end, we need to
compute quotients of the Bolza order $\qb$.

\begin{remark}\label{pres}
In \Tref{43b} we obtained the presentation
\[
\qb = O_K\big[\alpha,\beta \,|\, \alpha^2_{\phantom{I}}={{-1}}+\alpha,\,
\beta^2={{-1}}+\beta, \,
\beta\alpha=({{-1}}-\sqrt{2})+\alpha+\beta-\alpha\beta\big].
\]
%% This is true as well:
% Adjoining $\frac{1}{3}$, we can take $\beta = \beta' + \frac{1+2\sqrt{2}}{3} \alpha +\frac{1-\sqrt{2}}{3}$, and then
%$$\qb[\frac{1}{3}] = O_K[\frac{1}{3}][\alpha,\beta' \,|\, \alpha^2={{-1}}+\alpha,\, \beta'^2=\frac{\sqrt{2}}{3}, \, \beta'\alpha+\alpha\beta' =\beta'].$$
\end{remark}

The symplectic involution $z \mapsto z^*$ on the quaternion algebra $D$ (of \eq{D_MR}) is defined by $i^* = -i$ and $j^* = -j$. It follows from the definition of $\alpha,\beta$ in \eq{a_re} and \eq{b_re} that
\begin{equation}\label{thisisinv}
\alpha^* = 1 - \alpha,\qquad  \beta^* = 1-\beta;
\end{equation}
so in particular the order $\qb$ is preserved under the involution. This is particularly useful for the computation of the groups, because the norm is defined by $N(x) = xx^*$ for every $x \in D$.

\subsection{The Bolza order modulo $2$}

\def\qbb{{\overline{\qb}}}     % The order mod 2.        (ring of order 256)
\def\qbbb{{\widetilde{{\qb}}}} % The order mod sqrt{2}.  (ring of order 16)

Let us compute the ring $\qbb = \qb/2\qb$, which will be used below to compute the index of $\qb^1(2)$ in $\qb^1$.

Notice that $O_K/2O_K = \Z[\sqrt{2}]/2\Z[\sqrt{2}] = \F_2[\epsilon \,|\, \epsilon^2 = 0]$, where $\epsilon$ stands for the image of $\sqrt{2}$ in the quotient ring.
\begin{prop}
$\qbb = \qb/2\qb$ is a local noncommutative ring with $256$ elements, whose residue field has order $4$, and whose maximal ideal $J$ has nilpotency index $4$. Moreover each of the quotients $J/J^2$, $J^2/J^3$ and $J^3 = J^3/J^4$ is one-dimensional over $\qbb/J \,\cong\, \F_4$.
\end{prop}
\begin{proof}
Replacing $\beta$ by $\beta' = \beta+\alpha+1+\epsilon$ in the presentation of \Rref{pres}, we obtain the quotient
$$\qbb = \F_2[\epsilon \,|\, \epsilon^2 = 0][\alpha,\beta' \,|\, \alpha^2=1+\alpha,\, \beta'^2   = \epsilon, \, \beta'\alpha+\alpha\beta'=\beta'],$$
where $\epsilon$ is understood to be central (which actually follows from the relations).

This ring has a maximal ideal $J = \beta'\qbb$, with $J^2 =
\epsilon\qbb$ and $J^3 = \epsilon \beta'\qbb$, and with a quotient
ring
\[
\qbb/J = \F_2[\alpha \,|\, \alpha^2 = 1+\alpha] \,\cong\,
\F_4.
\]
Taking $\F_4 = \F_2[\alpha] = \F_2+\F_2\alpha$, we obtain
$$
\qbb = \F_4 \oplus \F_4 \beta' \oplus \F_4 \epsilon \oplus \F_4
\epsilon\beta',
$$ where $\beta'$ acts on $\F_4$ by $\beta' \alpha =
(\alpha+1)\beta'$, $\beta'^2 = \epsilon$ and $\epsilon^2 = 0$, so the
ring has $256$ elements.
\end{proof}

\subsection{The quotients $\qbbb = \qb/\sqrt{2}\qb$}

Since $\epsilon$ stands for the image of $\sqrt{2}$ in $\qbb$, we immediately obtain the quotient $\qb/\sqrt{2}\qb = \qbb / \epsilon \, \qbb$:
\begin{prop}\label{modsq2}
$\qbbb = \qb/\sqrt{2}\qb$ is a local noncommutative ring with a maximal ideal with $4$ elements and a quotient field of order $4$.
\end{prop}
\begin{proof}
Taking $\epsilon = 0$ in the presentation of $\qbb = \qb/2\qb$ obtained above, we get
$$\qbbb  = \F_2[\alpha,\beta' \,|\, \alpha^2=1+\alpha,\, \beta'^2   = 0, \, \beta'\alpha+\alpha\beta'=\beta'],$$
which can be written as
$$\qbbb = \F_4 \oplus \F_4 \beta';$$
this quotient of $\qbb = \qb/2\qb$ has $16$ elements.
The ideal
\[
\beta'\qbbb= \F_2\beta'+\F_2\alpha \beta'
\]
has four elements, and $(\beta'\qbbb)^2 = 0$.
%the quotient field is $\F_2[\alpha] = \F_4$, with $4{{-1}}=3$ elements of norm $1$.
\end{proof}

\subsection{Involution and norm}

The involution defined on $\qb$ clearly preserves $2\qb$, so it induces an involution on the quotient $\qbb$. Using \eq{thisisinv}, we conveniently have that $\beta'^* = \beta^*+\alpha^*+1+\epsilon = \beta'$.

The subring $\F_2[\epsilon,\alpha]$ of $\qbb$ is commutative, and the involution induces the automorphism $\s$ of $\F_2[\epsilon,\alpha]$  defined by  $\s(\alpha) = \alpha+1$ and $\s(\epsilon) = \epsilon$. The norm defined above coincides with the Galois norm,
$$N(x_0+x_1\alpha) = (x_0+x_1\alpha)(x_0+x_1(\alpha+1)) = x_0^2+x_0x_1+x_1^2$$
for $x_0,x_1 \in \F_2[\epsilon]$.
Furthermore, writing
$$\qbb = \F_2[\epsilon,\alpha] \oplus \F_2[\epsilon,\alpha]\beta',$$
we have for $y_0,y_1 \in \F_2[\epsilon,\alpha]$ that $(y_0+y_1 \beta')^* = y_0^*+\beta'y_1^* = y_0^*+y_1\beta'$.
Therefore, for every $y_0,y_1 \in \F_2[\epsilon,\alpha]$,
$$N(y_0+y_1\beta') = (y_0+y_1\beta')(y_0^*+y_1\beta') = N(y_0) + N(y_1)\epsilon \in \F_2[\epsilon].$$
Together, we have
$$N(x_{00}+x_{01}\alpha+x_{10}\beta'+x_{11}\alpha\beta') = (x_{00}^2+x_{00}x_{01}+x_{11}^2)+(x_{10}^2+x_{10}x_{11}+x_{11}^2)\epsilon$$
for every $x_{00},x_{01},x_{10},x_{11} \in \F_2[\epsilon]$.

Clearly, an element is invertible if and only if its norm is invertible. There are two invertible elements in $\F_2[\epsilon]$, namely $1$ and $1+\epsilon$, and $1+\epsilon = N(1+\epsilon\alpha)$ is obtained as a norm, so we conclude:
\begin{cor}\label{hasindex2}
The subgroup $\qbb^1 = \set{ x \in \qbb \,:\, N(x) = 1}$ has index $2$ in the group of invertible elements
$\mul{\qbb}$.
\end{cor}

In contrast, when we reduce further to the quotient
$\qbbb = \qb/\sqrt{2}\qb$, which is equal to
$\qbb/\epsilon \qbb$, the induced norm function
takes values in $\F_2[\epsilon]/\epsilon\F_2[\epsilon] = \F_2$, where only
the identity is invertible.  We therefore obtain the following
corollary.
\begin{cor}\label{hasindex1}
The subgroup $\qbbb^1 = \set{ x \in \qbbb \,:\, N(x) = 1}$ is equal to $\mul{\qbbb}$.
\end{cor}

\subsection{Subgroups of $\mul{\qbb}$}

The ring $\qbb = \qb / 2\qb$ has a unique maximal ideal $J = \beta'\qbb$, and its powers are
\[0 = J^4 \subset J^3 = \epsilon\beta'\qbb \subset J^2 = \epsilon\qbb \subset J = \beta'\qbb.
\]
Similarly to congruence subgroup of $\qb$, for every ideal $I \normali \qbb$ which is stable under the involution (so that the involution and thus the norm are well defined on the quotient $\qbb/I$), we have the subgroups
$$\qbb^1(I) = \qbb^1 \cap (1+I)$$
and
$$\qbb^\times(I) = \qbb^\times \cap (1+I);$$
when $I = x\qbb$, we write $\qbb^1(x)$ and $\qbb^\times(x)$ for $\qbb^1(x\qbb)$ and $\qbb^\times(x\qbb)$, respectively.

\begin{prop}
The numbers along edges in Figure~\ref{sgoQ} are the relative
indices of the depicted subgroups.
\end{prop}
\begin{proof}
The argument leading to \Cref{hasindex2} also implies that
$$\dimcol{\qbb^\times(\beta')}{\qbb^1(\beta')} = \dimcol{\qbb^\times(\epsilon)}{\qbb^1(\epsilon)} = 2,$$
because the invertible element $1+\epsilon\alpha$, whose norm is $1+\epsilon$ and not $1$, is in $\qbb^\times(\epsilon)$. However,
$$\qbb^\times(\epsilon\beta') = \qbb^1(\epsilon\beta')$$
because $N(1+x_3 \epsilon \beta') = 1$ for every $x_3 \in \F_2[\alpha]$. Moreover, since $\qbb$ is explicitly known, it is easy to compute the quotients
$$\qbb^\times/ \qbb^\times(\beta') \cong \F_4^\times$$
and
$$\qbb^\times(J^i)/ \qbb^\times(J^{i+1}) \cong \F_4^+, \qquad (i=1,2,3);$$
together, we have all the indices of the subgroups as depicted in the diagram.
\end{proof}

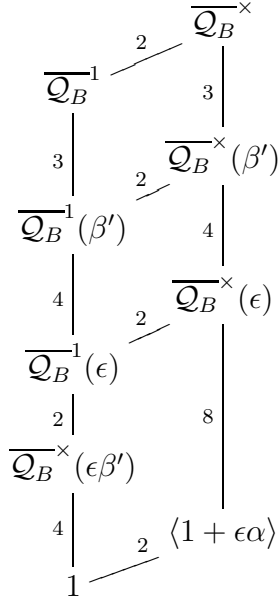
\begin{figure}
$$\xymatrix@R=4pt@C=4pt{
{} & \qbb^\times \ar@{-}[ld]_2 \ar@{-}[dd]_3 \\
\qbb^1 \ar@{-}[dd]_3 & {} \\
{} & \qbb^\times(\beta') \ar@{-}[ld]_2 \ar@{-}[dd]_4 \\
\qbb^1(\beta') \ar@{-}[dd]_4 & {} \\
{} & \qbb^\times(\epsilon) \ar@{-}[ld]_2 \ar@{-}[dddd]_8 \\
\qbb^1(\epsilon) \ar@{-}[dd]_2 & {} \\
{} & {} \\
\qbb^\times(\epsilon\beta') \ar@{-}[dd]_4 & {} \\
{} & \sg{1+\epsilon \alpha} \ar@{-}[dl]_2 \\
1 & {}
}$$
%{} & \qbb^\times(\epsilon\beta') \ar@{-}[r] & \qbb^\times(\epsilon) \ar@{-}[r] &  \qbb^\times(\beta') \ar@{-}[r]&  \qbb^\times
%\\
%1 \ar@{-}[r] & \qbb^1(\epsilon\beta') \ar@{-}[r] \ar@{-}[u] & \qbb^1(\epsilon) \ar@{-}[r] \ar@{-}[u] & \qbb^1(\beta')\ar@{-}[r] \ar@{-}[u] &  \qbb^1 \ar@{-}[u]
\caption{Subgroups of $\qbb^\times$, with relative indices}\label{sgoQ}
\end{figure}

Since we encounter several small classical groups, let us record their interactions.
\begin{remark}\label{smallgp}
The group $A_4$ of even permutation on $4$ letters is isomorphic to $\PSL[2](\F_3)$, and $S_4 \cong \PGL[2](\F_3)$.
The group $A_4$ has two central extensions by $\Z/2\Z$: the trivial one, namely $A_4 \times \Z/2\Z$, and the group $\SL[2](\F_3)$. Likewise
%
%\footnote{As far as GL(2,3) is concerned, we could write to some experts to find a place in the MODERN literature where it is proved that the symmetry group of Bolza is GL(2,3).  Sabourau and I have been mostly interested in the metric properties of the surface so I have to admit we have not looked up the symmetry group as we should have. One possibility is to write to Kallel and/or his coauthors who mention in their paper that Bolza himself was aware of this over 100 years ago.}
%
$\GL[2](\F_3)$ is a central extension of
$S_4$ by $\Z/2\Z$, and we have the short exact sequences
$$\xymatrix{
1 \ar@{->}[r] & \Z/2\Z \ar@{->}[r] \ar@{=}[d] & \GL[2](\F_3) \ar@{->}[r] & \PGL[2](\F_3) \cong S_4 \ar@{->}[r] & 1 \\
1 \ar@{->}[r] & \Z/2\Z \ar@{->}[r] & \SL[2](\F_3) \ar@{->}[r] \ar@{^(->}[u] & \PSL[2](\F_3) \cong A_4 \ar@{->}[r] \ar@{^(->}[u] & 1
}$$
where the image of $\Z/2\Z$ in both groups is central.

Since $A_4$ has the triangle group presentation
\[
\Delta_{(3,3,2)} \cong \sg{x,y \,|\, x^3 = y^3 = (xy)^2 = 1},
\]
it follows that $\SL[2](\F_3)$ can be presented as
$$\sg{x, y \,|\, x^3 = y^3 = (xy)^4 = 1, \, [x,(xy)^2] = 1}.$$
%$$\sg{x, y, t \,|\, x^3 = y^3 = t^2 = 1, \, t = (xy)^2,\, [x,t] = [y,t] = 1}.$$
\end{remark}

%As usual, we denote by $A_4$ the group of even permutations of order $4$, so that $\card{A_4} = 12$.
\begin{prop}\label{someiso}
The following holds for the quotients of $\qbb^1$: % {\ }
\begin{eqnarray}
\label{si0}
\qbb^1 / \qbb^1(\epsilon\beta') & \cong & \SL(\F_3),\\
\label{si1}
\qbb^1 / \qbb^1(\epsilon) & \cong & A_4.
\end{eqnarray}
%\begin{enumerate}
%\item\label{si0} The group $\qbb^1 / \qbb^1(\epsilon\beta')$ is isomorphic to $\SL(\F_3)$.
%\item\label{si1} The group $\qbb^1 / \qbb^1(\epsilon)$ is isomorphic to $A_4$.
%\end{enumerate}
\end{prop}
\begin{proof}
%\begin{enumerate}
%\item
The elements $\alpha, \beta \in \qb^1$, which satisfy $\alpha^3 = \beta^3 = {{-1}}$, map to their images $\alpha,\beta \in \qbb^1$. In $\qbb^1$ we have the relations $\alpha^3 = \beta^3 = 1$ (noting that ${{-1}} =  1$ in $\qbb = \qb / 2\qb$), and also, by computation, $(\alpha\beta)^2 = 1+\epsilon+ \epsilon \alpha \beta'$. Passing to the quotient $\qbb^1/\qbb^1(\epsilon\beta')$, we have that
    $$\alpha^3 = \beta^3 = (\alpha\beta)^4 = [\alpha,(\alpha\beta)^2] = [\beta,(\alpha\beta)^2] = 1$$
    since in this quotient $(\alpha\beta)^2 = 1+\epsilon$, which is central of order $2$. By \Rref{smallgp}, the group with this presentation is $\SL[2](\F_3)$, of order $24$. %  since modulo $(\alpha\beta)^2$ it becomes the triangle group $\Delta_{3,3,2}$, which is a well known presentation of $A_4$.
    To complete the proof, it remains to show that the image of $\sg{\alpha,\beta}$ in $\qbb^1/\qbb^1(\epsilon\beta')$ has order $24$. This can be done by computing in each quotient separately: \begin{itemize}\item $\alpha$ generates $\qbb^1/\qbb^1(\beta') \cong \Z/3\Z$; \item $\alpha\beta = 1+\epsilon\alpha+\alpha\beta' \equiv 1+\alpha\beta' \pmod{\qbb^1(\epsilon)}$, and $\beta\alpha = 1+\epsilon\alpha + (1+\alpha)\beta' \equiv 1 +(1+\alpha)\beta'$, which together generate $\qbb^1(\beta')/\qbb^1(\epsilon)$, isomorphic to $\Z/2\Z \times \Z/2\Z$;
    \item and $(\alpha\beta)^2$ generates $\qbb^1(\epsilon)/\qbb^1(\epsilon\beta') \cong \Z/2\Z$ as we have seen.
    \end{itemize}
%\item
As for the second isomorphism, passing with the previous item to the quotient $\qbb^1/\qbb^1(\epsilon)$, we have that
    $$\alpha^3 = \beta^3 = (\alpha\beta)^2 = 1,$$
    so we get the triangle group $\Delta_{(3,3,2)} \cong A_4$.
(An explicit isomorphism is obtained by $\alpha \mapsto (123)$ and $\beta \mapsto (124)$).
 % This also works but I did not check for consistency: we obtain the isomorphism $(12)(34) \mapsto 1 + \beta'$ and $(123) \mapsto \alpha$. % 1+\beta',\alpha: checked.
%\end{enumerate}
\end{proof}

\begin{remark}
The following quotients decompose as direct products:
$$\qbb^\times / \qbb^\times(\epsilon\beta') \,\cong\, \SL(\F_3) \times (\Z/2\Z),$$
$$\qbb^\times / \qbb^1(\epsilon) \,\cong\, A_4 \times (\Z/2\Z).$$
\end{remark}
\begin{proof}
The element $1+\epsilon \alpha$, which has order $2$, is not in $\qbb^1$ because it has norm $1+\epsilon$. To prove the first isomorphism, it suffices to note that $1+\epsilon \alpha$ commutes with the generators $\alpha$, $\beta = 1+\epsilon+\alpha+\beta'$ of $\qbb^1/\qbb^1(\epsilon\beta')$, because
\[
(1+\epsilon\alpha)(1+\alpha+\epsilon+\beta')(1+\epsilon\alpha)^{{{-1}}} =
1+\alpha+\epsilon+\beta'+\epsilon\beta' \equiv
1+\alpha+\epsilon+\beta'.
\]

The second isomorphism follows by taking the first one modulo
$\qbb^1(\epsilon)$, giving the quotient $\qbb^1/\qbb^1(\epsilon)$,
which is isomorphic to $A_4$ by \Pref{someiso}.%
\end{proof}

\begin{remark}\label{exp2}
The group $\qbb^\times(\epsilon)$ is isomorphic to $(\Z/2\Z)^4$.
\end{remark}
\begin{proof}
By definition, $\qbb^\times(\epsilon) = 1+\F_2[\alpha]\epsilon+\F_2[\alpha]\beta'\epsilon$ has order $16$. But for every $f \in \qbb$, $(1+f\epsilon)^2 = 1+2f\epsilon+f^2\epsilon^2 = 1$. This shows that the group has exponent $2$, so it is abelian.
\end{proof}

\section{The Bolza group as a congruence subgroup}\label{sec:13}

Our goal is to compare the Fuchsian group $B$, corresponding to the
Bolza surface, to congruence subgroups of $\qb^1$ modulo $\set{\pm
1}$.  To simplify notation, we write
\[
\sqb[] = \qb^1/\set{\pm 1}
\]
and
\begin{equation}
\label{signdef}
\sqb[I] = \sg{{{-1}},\qb^1(I)}/\set{\pm 1}
%
%\sqb[I] = \qb^1(I)/\set{\pm 1}
%
\end{equation}
for any ideal $I \normali \Z[\sqrt{2}]$.

By \Lref{112}, the group $B \sub \sqb[]$ is generated, as a normal
subgroup, by the element $\delta = (\alpha\beta)^2(\alpha^2\beta^2)^2
= 1+\sqrt{2} (1 + (1+\sqrt{2})(\alpha - \beta))$.

\begin{prop}\label{si2}
The map $\qb^1/\qb^1(\sqrt{2}) \ra \qbb^1/\qbb^1(\epsilon)$, induced
by the projection $\qb \ra \qbb$, is an isomorphism.
\end{prop}
\begin{proof}
The projection modulo $\sqrt{2}$ provides an injection
$$\qb^1/\qb^1(\sqrt{2}) \ra \qbb^\times/\qbb^1(\epsilon),$$ which a priori need not be onto $\qbb^1/\qbb^1(\epsilon)$, even taking into account that every element of $(\qb/2\qb)^{\times}$ has norm $1$. But in \Pref{someiso} we observed that the images of $\alpha,\beta \in \qb^1$ generate $\qbb^1/\qbb^1(\epsilon)$.
%  In \Sref{sec:7} we observed that $\alpha$ and $\beta =1+\alpha+\beta'$ (equality modulo $\epsilon$) are in $\qb^1$. But in $\qbbb$ we have that $(1+\beta')\alpha^2(1+\beta')^{{{-1}}} = 1+\alpha+\beta' = \beta$, and we have just seen that the images of $\alpha$ and $\beta$, namely $(123)$ and $(12)(34)(123)^2(12)(34) = (124)$, generate the whole group.
\end{proof}

\begin{thm}\label{main1}
The Bolza group $B$ satisfies $\sqb[2] \,\subset\, B \,\subset\, \sqb[\sqrt{2}]$, and
$$\sqb/B \cong \qbb^1/\qbb^1(\epsilon\beta').$$
\end{thm}
\begin{proof}
Noting that ${{-1}} \in \qb^1(2)$, we investigate the chain of groups
$$\sqb[2] \,\sub\, B\sqb[2] \,\sub\, \sqb[\sqrt{2}] \,\sub\, \sqb.$$

Let $\phi \co \qb^1 \ra \qbb^1$ be the map induced by the projection $\qb \ra \qbb = \qb/2\qb$. This homomorphism, whose kernel is $\qb^1(2)$, is well defined on $\sqb = \qb^1/\set{\pm 1}$, since ${{-1}} \in \qb^1(2)$. Furthermore, $\phi$ carries $\sqb$
onto $\qbb^1$, and the subgroup $\sqb[\sqrt{2}]$ onto $\qbb^1(\epsilon)$, by \Pref{si2}.

At the same time, because $\phi(\delta) = 1+\epsilon\beta' \in \qbb^{1}(\epsilon\beta')$, the normal subgroup it generates is mapped into $\qbb^1(\epsilon\beta')$. This proves that
$$\dimcol{\sqb}{B \cdot \sqb[2]} = \dimcol{\qbb^1}{\qbb^1(\epsilon\beta')} = 24.$$
But since $\sqb$ is isomorphic to $\Delta_{(3,3,4)}$, we have by \Pref{102}
that
\[
\dimcol{\sqb}{B} = 24
\]
as well. This proves that $B = B \sqb[2]$, so that $\sqb[2] \,\sub\, B$. It follows that the injection of $\qb^1/\qb^1(2)$  into $\qbb^1$ sends $\delta$ to $1+\epsilon\beta'$, and the normal subgroup $B$ generated by the former, to the normal subgroup $\qbb^1(\epsilon\beta')$ generated by the latter.
\end{proof}

\def\Sym{{\operatorname{Sym}}}

Let $\Sym_{(3,3,4)}(B)$ denote the quotient $\sqb/B$, which is the group of orientation preserving symmetries of the Bolza surface stemming from the $(3,3,4)$ tiling. % Where is $\Sym_{(2,3,8)(B)$ in all this?
\begin{cor}
% \footnote{We have to be a bit careful here because literally the automorphism group is somewhat larger than $Q^1/B$ because the Bolza surface is actually in the (2,3,8) family, so there is supposed to be an extra automorphism that's invisible from the viewpoint of the quaternionic structure.}
%
The symmetry group $\Sym_{(3,3,4)}(B)$ is isomorphic to $\SL[2](\F_3)$.
\end{cor}
\begin{proof}
Indeed, the automorphism group $\sqb/B \cong \qbb^1/\qbb^1(\epsilon\beta')$ was computed in \Pref{someiso}.\eq{si0}.
\end{proof}

Let us add this result to the observations made in \Sref{sec:238}, where we embedded
$$\Delta_{(3,3,4)} = \sg{\alpha, \beta \,|\, \alpha^3 = \beta^3 =
(\alpha\beta)^4 = 1}$$ as a subgroup of index $2$ in
$$\Delta_{(2,3,8)} = \sg{x, y \,|\, x^2 = y^3 = (xy)^8 = 1}$$
via the map $\alpha \mapsto y$ and $\beta \mapsto xyx$.
Since $B \sub \Delta_{(3,3,4)}$ is the normal subgroup generated by $(\alpha \beta)^2(\alpha^2\beta^2)^{-2}$ by \Lref{112}, its image in $\Delta_{(2,3,8)}$ is $\sg{(y x)^4(y^{{{-1}}}x)^4}^{\sg{y,xyx}}$, which happens to be normal in $\Delta_{(2,3,8)}$, and the quotient group is
$$\sg{x, y \,|\, x^2 = y^3 = (xy)^8 =(y x)^4(y^{{{-1}}}x)^4 = 1}.$$
This quotient is isomorphic to $\GL[2](\F_3)$ by taking $x \mapsto \smat{0}{1}{1}{0}$ and $y \mapsto \smat{1}{1}{0}{1}$.
\begin{cor}\label{133}
The symmetry group $\Sym_{(2,3,8)}(B) = \Delta_{(2,3,8)}/B$ is isomorphic to $\GL[2](\F_3)$.
\end{cor}

We can also compute the quotient of $B$ modulo the principal congruence subgroup it contains:

\begin{remark}
We have that $B/\sqb[2] \,\cong\, \Z/2\Z \times \Z/2\Z$.
\end{remark}
\begin{proof}
Indeed, $\qbb^1(\epsilon\beta')$ has order $4$, and as a subgroup of $\qbb^{\times}(\epsilon)$, which is of exponent $2$ by \Pref{exp2}, we obtain
$$B/\sqb[2] \,\cong\, \qbb^1(\epsilon\beta') \,\cong\, \Z/2\Z\times \Z/2\Z,$$
as claimed.
\end{proof}

\begin{cor}
$\sqb[\sqrt{2}]$ is generated by $B$ and the torsion element $\varpi$
of~\eqref{varpidef}.
\end{cor}
\begin{proof}
As we have seen before, $B$ is torsion free, so $\varpi \not \in B$,
and $\sg{B,\varpi}$ strictly contains $B$, so the result follows from
$\dimcol{\qb^1(\sqrt{2})}{\qb^1(2)} = 8$.
\end{proof}

%The congruence subgroups of quaternion algebras play an important role
%in the local Langlands correspondence; see e.g., Harris \&
%Taylor~\cite{HT}.

\section{Computations in the Bolza twins}
\label{twins}

In this section and the ones that follow, we will present some
explicit computations with the ``twin'' surfaces corresponding to the
algebraic primes factoring rational primes in $K = \Q(\sqrt{2})$.  Recall that $O_K = \Z [\sqrt{2}]$.  We
first state a result on quotients co-prime to $6$, which follows from
the definition of ramification (and splitting) in a quaternion
algebra.

\begin{lemma} \label{lem:matrixlemma}
Let $I \normali O_K$ be a prime ideal. If $2$ and $3$ are invertible
modulo~$I$, then $\qb/I\qb \cong \M_2(O_K/I)$.
\end{lemma}
\begin{proof}
It is convenient to make the substitution
\[
\alpha = \alpha'+\frac{1}{2},\quad \beta =
\beta'+\frac{1+2\sqrt{2}}{3}\alpha'+\frac{1}{2}.
\]
Then $\alpha'$ and $\beta'$ anticommute, and we obtain a standard
presentation
\[
\qb/I\qb = (O_K/I)\left[\alpha',\beta' \,\left|\,
\alpha'^2=-\tfrac{3}{4},\ \beta'^2=\tfrac{\sqrt{2}}{3},\
\beta'\alpha' = -\alpha'\beta'\right.\right].
\]
Since $\alpha'^2$ and $\beta'^2$ are invertible, the
quotient~$\qb/I\qb$ is a quaternion algebra over~$O_K/I$.  But as a
finite integral domain, $O_K/I$ is a field, so by Wedderburn's little
theorem (that the only finite division algebras are fields),
$\qb/I\qb$ is necessarily isomorphic to $\M_2(O_K/I)$.
\end{proof}

%\subsection{Groups}

% We are interested in subgroups of $\qb^1/\set{\pm 1}$. Since the ideals in this section are odd, we must amend \eq{signdef}, and define \begin{equation}\label{signdef2} \sqb[I] = \sg{{{-1}}, \qb^1(I)}/\set{\pm 1}. \end{equation}

%We computed above that $\qb/I\qb \cong \M_2(O_K/I)$ for every prime ideal $I \normali \Z[\sqrt{2}]$ which is prime to $2,3$.  Hence
%\[
%\qb^1/\qb^1(I) \cong \SL(O_K/I)
%\]
\begin{corollary} \label{cor:primefactorcor}
Let $I \normali O_K$ be a prime ideal such that $6$ is invertible modulo $I$.  Then
\[
\sqb / \sqb(I) \cong \PSL(O_K / I).
\]
\end{corollary}
\begin{proof}
By Lemma~\ref{lem:matrixlemma} and strong approximation we have $\qb^1/\qb^1(I) \cong \SL(O_K/I)$.
\end{proof}

%We therefore obtain the following corollary.
%\begin{corollary}
%We have
%\[
%\qb^1/\sg{{{-1}},\qb^1(I)}\cong \PSL(O_K/I)
%%
%% By $\qb^{\pm}(I)$ I mean the subgroup generated by $\qb^{1}(I)$ and $\pm 1$.
%% \qb^{\pm}/\sqb[I] \cong \PSL(O_K/I) % By $\qb^{\pm}(I)$ I mean the subgroup generated by $\qb^{1}(I)$ and $\pm 1$.
%\]
%for every ideal $I$ such that $2$ and $3$ are invertible modulo $I$.
%\end{corollary}

\begin{lemma} \label{lem:normalsubgroups}
Let $p$ be a rational prime that splits in $K$, so that $p O_K = I_1 I_2$ for distinct prime ideals $I_1, I_2 \normali O_K$.  There are exactly two normal subgroups $H \normali \sqb$ such that $\sqb / H \cong \PSL(\mathbb{F}_p)$, namely $\sqb(I_1)$ and $\sqb(I_2)$.
\end{lemma}
\begin{proof}
Recall that the rational primes $p$ splitting in $K$ are precisely those satisfying $p \equiv \pm 1 \mbox{  } \mathrm{mod}\mbox{ } 8$.  Let $H \normali \sqb$ be a normal subgroup such that $\sqb / H \cong \PSL(\mathbb{F}_p)$ and choose a surjection $\varphi: \sqb \to \PSL(\mathbb{F}_p)$ such that $\ker (\varphi) = H$.  Clearly $\varphi$ is determined by the triple $(\varphi(\alpha), \varphi(\beta), \varphi(\alpha \beta)^{{{-1}}}) \in (\PSL(\mathbb{F}_p))^3$, where we use $\alpha, \beta$ to denote the images of these elements in $\sqb$.  Note that $\varphi(\alpha)$ must have order $3$; otherwise it would be trivial and $\varphi(\sqb)$ would be abelian, contradicting surjectivity.  Similarly, $\varphi(\beta)$ has order $3$.  Since $\alpha \beta$ has order $4$ in $\sqb$, the order of $\varphi(\alpha \beta)$ must divide $4$.  If $\varphi(\alpha \beta)$ is trivial, then again we get that $\varphi(\sqb)$ is abelian.  If $\varphi(\alpha \beta)^2$ is trivial, then it is easy to show that $\langle\varphi(\alpha \beta), \varphi(\beta \alpha) \rangle \subseteq \PSL(\mathbb{F}_p)$ is a normal subgroup and hence all of $\PSL(\mathbb{F}_p)$ since the latter is a simple group.  However, this is again absurd because a finite group generated by two involutions must be dihedral.  Thus the order of $\varphi(\alpha \beta)$ is $4$.

We have thus shown that $(\varphi(\alpha), \varphi(\beta), \varphi(\alpha \beta)^{{{-1}}})$ is a non-exceptional group triple in the sense of~\cite[Section 8]{CV15}.  Moreover, since none of these three elements of $\PSL(\mathbb{F}_p)$ can be a scalar matrix, it follows that if $(g_1, g_2, g_3) \in (\PSL(\mathbb{F}_p))^3$ is any triple such that $g_1 g_2 g_3 = 1$ and $(\tr (g_1), \tr(g_2), \tr(g_3)) = (\tr \varphi(\alpha), \tr \varphi(\beta), \tr \varphi(\beta^{{{-1}}} \alpha^{{{-1}}}))$, then the orders of $g_1, g_2, g_3$ are $3, 3, 4$, respectively.  In particular, the subgroup $\langle g_1, g_2, g_3 \rangle$ is never abelian.  Hence the trace triple
\[
(\tr \varphi(\alpha), \tr \varphi(\beta), \tr \varphi(\beta^{{{-1}}} \alpha^{{{-1}}}))
\]
is not commutative, and so it must be projective by~\cite[Theorem 4]{Ma67}.
Since $p \not\in \{ 2, 3 \}$, all the hypotheses of~\cite[Proposition 8.10]{CV15} hold.  By that proposition, there are at most two normal subgroups $H$ such that $\sqb / H \cong \PSL(\mathbb{F}_p)$.  On the other hand, $\sqb(I_1)$ and $\sqb(I_2)$ are clearly distinct and satisfy this condition by Corollary~\ref{cor:primefactorcor}.  We are grateful to J.~Voight for directing us to the reference~\cite{CV15}.
\end{proof}

\begin{remark}
For every $p$ splitting in $K$, we obtain a pair of Bolza twin surfaces $M$ of
genus~$g(M)=\frac{p(p^2 - 1)}{48} + 1$, i.e., Euler characteristic
$\chi(M)=-\frac{p(p^2 - 1)}{24}$, and area $\frac{\pi p(p^2 {{-1}})}{12}$.
Since the area of the (3,3,4) triangle is $\frac{\pi}{12}$ and its
double is $\frac{\pi}{6}$, the automorphism group generated by
orientation-preserving elements of the triangle group has order
$\frac{p(p^2 {{-1}})}{2}$, namely that of $\PSL(\F_p)$.
\end{remark}

\section{Bolza twins of genus $8$}

We now apply the results of Section~\ref{twins} to the pair
$(1+2\sqrt{2})O_K$ and $(1-2\sqrt{2})O_K$.  Both of these have
norm~$7$, and therefore their principal congruence quotients are
isomorphic to the group $\PSL(\F_7)$ of order 168.  These ideals give
rise to twin surfaces analogous to the Hurwitz triplets (see
\cite{KSV2}), namely non-isometric surfaces with the same automorphism
group.  Note that by estimate~\eqref{81c}, these Fuchsian groups
contain no elliptic elements.

The normal subgroup of the triangle group generated by each of these
in \magma{} produces a presentation with 16 generators and a single
relation of length 32, corresponding to Fuchsian groups of a Riemann
surface of genus~$8$.  Therefore it coincides with the corresponding
congruence subgroup, since it gives the correct order of the symmetry
group (i.e., index in the (3,3,4) triangle group), namely order 168.

To find these groups, we searched for subgroups of index 168 using
\magma, and looked for the simplest generator whose normal closure is
the entire group.  The results are summarized in Lemmas~\ref{l151} and
\ref{l152} below.

The numerical values reproduced below suggest that the systole of the
surface corresponding to the ideal $(1+2\sqrt{2})O_K$ should be
smaller than the systole of the surface corresponding to the
ideal~$(1-2\sqrt{2})O_K$.

\begin{lemma}
\label{l151}
The element $-(\alpha\beta^{{{-1}}})^4$ is in $\qb^1(1+2\sqrt{2})$. %
Its normal closure is the full congruence subgroup corresponding to
the ideal generated by $1+2\sqrt{2}$.
\end{lemma}

\begin{proof}
With respect to the module basis we have
\[
(\alpha\beta^{{{-1}}})^4=(5+3\sqrt{2})(\alpha-\alpha\beta)-(2+2\sqrt{2}),
\]
which is congruent to ${{-1}}$ modulo the ideal $(1+2\sqrt{2})$. On the other hand,
$$\sg{\alpha, \beta \,|\, \alpha^3 = \beta^3 = (\alpha\beta)^4 = (\alpha\beta^{{{-1}}})^4 = 1}$$
has order $168$, showing that the normal closure of $(\alpha\beta^{{{-1}}})^4$ is the full congruence subgroup.
%
%% Note: -(\alpha\beta^{{{-1}}})^4 =  (-7/2-2\sqrt{2})  + (1/2+\sqrt{2})/3*i + (5/2+3\sqrt{2}/2) * j + (-5/6-\sqrt{2}/2) * ij.
%%  which explains why the trace is $2x_0=-(7+4\sqrt{2})$.
%%  also $2(x_0{{-1}}) = -(9+4\sqrt{2}) = -(1+\sqrt(2))^2$!
\end{proof}

\begin{remark}
\label{t1}
The element $(\alpha\beta^{{{-1}}})^4$ has trace
$7+4\sqrt{2}=12.656\ldots{}$ Of the 16 generators produced by
\magma{}, 14 have this trace (up to sign), and the remaining two
generators have trace $19 + 13\sqrt{2}=37.384\ldots$ (up to sign).
The smaller value $7+4\sqrt{2}= 12.656\ldots$ is a good candidate for
the least trace of a nontrivial element for this Fuchsian group.
\end{remark}

\begin{lemma}
\label{l152}
The element $-(\beta^{{{-1}}}\alpha^{{{-1}}}\beta^{{{-1}}}\alpha\beta\alpha)^2$ is
in $\qb^1(1-2\sqrt{2})$.  Its normal closure is the full congruence
subgroup corresponding to the ideal generated by $1-2\sqrt{2}$.
\end{lemma}

\begin{proof}
A calculation shows that
\[
(\beta^{{{-1}}}\alpha^{{{-1}}}\beta^{{{-1}}}\alpha\beta\alpha)^2=
(7+5\sqrt{2})+(5+ 4\sqrt{2}) \alpha - (8 + 5\sqrt{2}) \beta
+(3+\sqrt{2})\alpha\beta.
\]
Adding $1$, the coefficients $8+5\sqrt{2}$, $5+4\sqrt{2}$ and $3+\sqrt{2}$ are divisible by $1-2\sqrt{2}$,
% The quotients are $4+3\sqrt{2}$, $3+2\sqrt{2}$ and $1+\sqrt{2}$.
so $(\beta^{{{-1}}}\alpha^{{{-1}}}\beta^{{{-1}}}\alpha\beta\alpha)^2$ is congruent to ${{-1}}$ modulo $1-2\sqrt{2}$ in the Bolza order.
Again, the normal closure is the full congruence subgroup because the group
$$\sg{\alpha, \beta \,|\, \alpha^3 = \beta^3 = (\alpha\beta)^4 = (\beta^{{{-1}}}\alpha^{{{-1}}}\beta^{{{-1}}}\alpha\beta\alpha)^2 = 1}$$
has order $168$ as well.
\end{proof}

\begin{remark}\label{t2}
The trace of $(\beta^{{{-1}}}\alpha^{{{-1}}}\beta^{{{-1}}}\alpha\beta\alpha)^2$ is
$9+6\sqrt{2}=17.485\ldots{}$ Of the 16 generators of the Fuchsian
group produced by \magma, 13 have this trace (up to sign), and the
remaining three have trace $14+11\sqrt{2}=29.556\ldots{}$ The smaller
value $9+6\sqrt{2}=17.485\ldots$ is a good candidate for the least
trace of a nontrivial element for this Fuchsian group.
\end{remark}

The traces in Remarks~\ref{t1} and~\ref{t2} can be compared to the
trace bound of \cite[Theorem~2.3]{KSV06a}, cf. \eq{81c}, which, since
$\qb \sub \frac{1}{6}O_K[i,j]$, gives for any ideal $I \normali
\Z[\sqrt{2}]$ and any $\pm 1 \neq x \in \qb^1(I)$
that~$\abs{\Tr_D(x)}> \frac{1}{4}\Norm(I)^2 - 2$.  In particular since
$\Norm(1+2\sqrt{2}) = \Norm(1-2\sqrt{2}) = 7$, we have for both
congruence subgroups mentioned in this section the trace lower bound
$\frac{41}{4} =10.25$.  Note that the trace appearing in
Remark~\ref{t1} exceeds the theoretical bound by less than 25\%.

\begin{remark}
It would be interesting to explore possible algorithms for the
computation of the systole of an explicitly given Fuchsian group,
possibly exploiting its fundamental domain using Voight \cite{Vo}.
\end{remark}

\section{Bolza twins of higher genus}
\label{s16}

In this section we collect explicit computations, performed in
\texttt{magma}, of Bolza twins for some primes $p > 7$ that split in
$K = \Q(\sqrt{2})$.  We briefly sketch the method.  Let $I_1$ and
$I_2$ be the two places of $K$ dividing~$p$.  We first obtain
presentations of the congruence subgroups $\sqb(I_1)$ and $\sqb(I_2)$.
By Lemma~\ref{lem:normalsubgroups}, these are the only two normal
subgroups of $\sqb$ such that the corresponding quotients are
isomorphic to $\PSL(\F_p)$.

The most efficient way to find such subgroups in practice is randomly
to generate a homomorphism from the triangle group onto $\PSL(\F_p)$.
Thus, we generate pairs $(A_1, A_2)$ of random elements of $\SL(\F_p)$
by means of the Product Replacement Algorithm and search for pairs
that generate $\SL(\F_p)$ and such that the projective orders of $A_1,
A_2, A_1 A_2$ are $3, 3, 4$, respectively.  Each such
pair corresponds to a surjection $\varphi: \sqb \to \PSL(\F_p)$ determined by $\varphi(\alpha) = \overline{A_1}$ and $\varphi(\beta) = \overline{A_2}$; here the bars denote images in $\PSL(\F_p)$.  We
search for two pairs such that the kernels of the corresponding
surjections are distinct; by Lemma~\ref{lem:normalsubgroups}, these
kernels are our two congruence subgroups.

This random search is far faster than any known deterministic algorithm.  Finding suitable pairs $(A_1, A_2)$ is very quick: for $p = 71$, for instance, a search through one million random pairs produced twenty suitable ones and took only a few seconds.

We then rewrite the presentations of these kernels by means of the
Reidemeister-Schreier algorithm, as implemented in \texttt{magma}; this is time-consuming, taking a few hours to run on a MacBook for $p = 71$.  It may be necessary to treat more than two surjections $\varphi$ before two different kernels are found.

In all cases that we have investigated, the Reidemeister-Schreier algortihm produces presentations
with $2 g_p$ generators and a single relation of length $4 g_p$; here
$g_p = p(p^2 - 1)/48 + 1$.  We search through this list for elements of minimal trace and for generators whose normal closure in $\sqb$ is the full
congruence subgroup and present our results below.  In some cases, none of the elements of minimal trace normally generate the entire congruence subgroup, and for one of the primes dividing $71$ we were unable to find any single element that normally generates the associated congruence subgroup.

\subsection{Bolza twins of genus $103$}

Factoring the rational prime $p=17$ as $-(1-3\sqrt{2})(1+3\sqrt{2})$,
we obtain a pair of Bolza twins of genus $103$, with automorphism
group $\PSL(\F_{17})$.  The order of $\PSL(\F_p)$ is $(p^2 - 1) p /2$.
This is 2448 for $p = 17$.  The element
\[
\alpha \beta\alpha^{{{-1}}} \beta^{{{-1}}}\alpha^{{{-1}}} \beta\alpha
\beta^{{{-1}}}\alpha \beta^{{{-1}}}\alpha^{{{-1}}} \beta^{{{-1}}}\alpha^{{{-1}}}
\beta\alpha \beta\alpha^{{{-1}}} \beta^{{{-1}}}\alpha \beta^{{{-1}}}\alpha \beta
\]
is congruent to 1 modulo $(1 - 3\sqrt{2})$ and normally generates the
corresponding congruence subgroup.  Its trace
is~$75+53\sqrt{2}\approx149.953\ldots$, which is the least trace (in
absolute value) among the $206$ generators (with a single relation of
length 412).  For the ``twin'' normal subgroup, we find a generator of
the form
\[
(\alpha \beta^{{{-1}}} \alpha^{{{-1}}} \beta^{{{-1}}} \alpha \beta \alpha^{{{-1}}} \beta^{{{-1}}} \alpha \beta)^2,
\]
equal to ${{-1}}$ mod $(1 + 3 \sqrt{2})$.  It generates the full
congruence subgroup, and gives the least trace,
namely~$79+56\sqrt{2}\approx158.195\ldots$, among all the generators.

\subsection{Bolza twins of genus $254$}

For $p=23$, there are two normal subgroups of the triangle group whose
quotient is $\PSL(\F_{23})$.  The order of $\PSL(\F_p)$
is~$(p^2{{-1}})p/2$.  This is 6072 for $p = 23$.  One obtains a generator
\[
\beta^{{{-1}}} \alpha \beta^{{{-1}}} \alpha^{{{-1}}} \beta
\alpha^{{{-1}}} \beta \alpha^{{{-1}}} \beta^{{{-1}}} \alpha
\beta^{{{-1}}} \alpha \beta \alpha ^{{{-1}}} \beta \alpha^{{{-1}}}
\beta \alpha
\]
with minimal trace of~$91+65\sqrt{2}$, whose normal closure is a group
with 508 generators and a single relation of length 1016.  This
generator is congruent to $+1$ modulo $5-\sqrt{2}$.

For its Bolza twin, the lowest trace appears to be~$119+84\sqrt{2}$.
An element that normally generates the congruence subgroup of
$5+\sqrt{2}$ is
\[
\alpha \beta \alpha \beta^{{{-1}}} \alpha \beta^{{{-1}}} \alpha
\beta^{{{-1}}} \alpha \beta \alpha^2 \beta^{{{-1}}} \alpha^{{{-1}}}
\beta \alpha \beta^{{{-1}}} \alpha \beta^{{{-1}}} \alpha
\beta^{{{-1}}}
\]
This generator is congruent to ${{-1}}$ modulo $5+\sqrt{2}$.

By Lemma~\ref{lem:normalsubgroups}, for each prime $p$ satisfying
$p\equiv\pm1 (\text{mod}\, 8)$, there are precisely two normal
subgroups of our triangle group with quotient isomorphic
to~$\PSL(\F_p)$, which are congruence subgroups corresponding to the
two algebraic primes factoring $p$.

\subsection{Bolza twins of genus 621}
Consider the decomposition $31=(9-5\sqrt{2})(9+5\sqrt{2})$.  The generator
\[
\beta  \alpha  \beta  \alpha^{{{-1}}}  \beta  \alpha  \beta^{{{-1}}}
\alpha  \beta^{{{-1}}}  \alpha  \beta  \alpha^{{{-1}}}  \beta^{{{-1}}}
\alpha^{{{-1}}}  \beta  \alpha^{{{-1}}}  \beta^{{{-1}}}  \alpha  \beta^{{{-1}}}
\alpha  \beta ^{{{-1}}}  \alpha ^{{{-1}}}
\]
is equivalent to $1$ modulo $9-5\sqrt{2}$ and normally generates the corresponding
principal congruence subgroup, producing a surface of genus $g=621 = 31(31^2-1)/48+1$.
This element has trace $153 + 109 \sqrt{2}$, which is the smallest
among the $2g$ generators.

For the Bolza twin, the element
\[
(\beta^{{{-1}}} \alpha \beta \alpha^{{{-1}}} \beta^{{{-1}}} \alpha \beta^{{{-1}}}
\alpha^{{{-1}}} \beta \alpha)^2
\]
equals ${{{-1}}}$ mod $(9 + 5 \sqrt{2})$, with normal closure with the
same properties, the least trace being $129 + 90 \sqrt{2}$.

\subsection{Bolza twins of genus 1436}
\label{s19}
Let $41=(7 - 2 \sqrt{2})(7 + 2 \sqrt{2})$.  For both $\sqb(7 - 2 \sqrt{2})$ and $\sqb(7 + 2 \sqrt{2})$, \texttt{magma} found presentations with $2g$ generators and a single relation of length $4g$, where $g = 1436 = 41 (41^2 - 1)/48 + 1$.  The generator
\[
\beta \alpha^{{{-1}}} \beta \alpha^{{{-1}}} \beta \alpha^{{{-1}}} \beta^{{{-1}}}
\alpha \beta^{{{-1}}} \alpha^{{{-1}}} \beta^{{{-1}}} \alpha \beta \alpha^{{{-1}}}
\beta \alpha^{{{-1}}} \beta^{{{-1}}} \alpha \beta \alpha \beta^{{{-1}}} \alpha
\beta \alpha^{{{-1}}}
\]
is congruent to ${{{-1}}}$ mod $(7 - 2 \sqrt{2})$ and has trace $208 \sqrt{2} + 295$.  Its normal closure is the full congruence subgroup $\sqb(7 - 2 \sqrt{2})$.

The pair of generators
\[
(\beta^{-1} \alpha)^{10}
\]
and
\[
\beta  \alpha  \beta  \alpha^{-1}  \beta  \alpha^{-1}  \beta^{-1}  \alpha  \beta^{-1}  \alpha^{-1}  \beta  \alpha  \beta^{-1}  \alpha^{-1}  \beta^{-1}  \alpha  \beta  \alpha^{-1}  \beta^{-1}  \alpha  \beta^{-1}  \alpha^{-1}  \beta  \alpha^{-1}
\]
are congruent to $-1$ and $1$ mod $(7 + 2 \sqrt{2})$, respectively, and they have traces $281 + 198 \sqrt{2}$ and $-(281 + 198 \sqrt{2})$, respectively.  The normal closure of this pair is the congruence subgroup $\sqb(7 + 2 \sqrt{2})$.  While $281 + 198 \sqrt{2}$ is the minimal trace among the $2g = 2872$ generators of $\sqb(7 + 2 \sqrt{2})$ found by \texttt{magma}, no single generator of this trace normally generates $\sqb(7 + 2 \
\sqrt{2})$.  It is, however, normally generated by the element
\[
\beta \alpha \beta \alpha^{{{-1}}} \beta \alpha^{{{-1}}} \beta \alpha^{{{-1}}}
\beta \alpha^{{{-1}}} \beta^{{{-1}}} \alpha \beta \alpha^{{{-1}}} \beta^{{{-1}}}
\alpha^{{{-1}}} \beta \alpha^{{{-1}}} \beta \alpha^{{{-1}}} \beta \alpha^{{{-1}}}
\beta \alpha \beta^{{{-1}}} \alpha^{{{-1}}},
\]
which is congruent to $1$ mod $(7 + 2 \sqrt{2})$ and has trace $681 + 481 \sqrt{2}$.

\subsection{Bolza twins of genus 2163}

Let $p = 47= (7 - \sqrt{2}) (7 + \sqrt{2})$.  The generator
\[
(\beta^{{{-1}}} \alpha \beta \alpha^{{{-1}}} \beta^{{{-1}}} \alpha \beta
\alpha^{{{-1}}} \beta^{{{-1}}} \alpha \beta^{{{-1}}} \alpha)^2
\]
has trace $529 + 374 \sqrt{2}$ and equals ${{{-1}}}$ mod $(7 - \sqrt{2})$,
while the generator
\[
\beta \alpha \beta \alpha^{{{-1}}} \beta \alpha \beta^{{{-1}}} \alpha^{{{-1}}}
\beta \alpha^{{{-1}}} \beta \alpha^{{{-1}}} \beta \alpha^{{{-1}}} \beta^{{{-1}}}
\alpha^{{{-1}}} \beta \alpha^{{{-1}}} \beta^{{{-1}}} \alpha \beta \alpha^{{{-1}}}
\beta \alpha^{{{-1}}} \beta \alpha^{{{-1}}}
\]
has trace $499 + 353 \sqrt{2}$ and equals $1$ mod $(7 + \sqrt{2})$.
In both cases, normal closures are subgroups with $2g$ generators and
a single relation of length $4g$, for $g = 2163$.

\subsection{Bolza twins of genus 7456}

Consider $p= 71=(11 + 5 \sqrt{2})(11 - 5 \sqrt{2})$.  The two elements
\[
 \beta \alpha \beta \alpha^{{{-1}}} \beta \alpha \beta^{{-1}} \alpha^{{-1}} \beta \alpha^{{-1}} \beta \alpha^{{-1}} \beta    \alpha \beta^{{-1}} \alpha^{{-1}} \beta^{{-1}} \alpha \beta \alpha \beta^{{-1}} \alpha^{{-1}}\beta \alpha^{{-1}} \beta\alpha^{{-1}}\beta\alpha^{{-1}}
\]
and
\[
\beta \alpha \beta \alpha^{{-1}} \beta \alpha^{{-1}} \beta \alpha
\beta^{{-1}} \alpha^{{-1}} \beta \alpha \beta^{{-1}} \alpha
\beta^{{-1}} \alpha^{{-1}} \beta^{{-1}} \alpha \beta \alpha^{{-1}} 
\beta \alpha \beta^{{-1}} \alpha^{{-1}} \beta \alpha \beta^{{-1}}
\alpha^{{-1}}
\]
are each congruent to $1$ modulo $(11 + 5 \sqrt{2})$, and they each have trace $-(951 + 672 \sqrt{2})$.  The normal closure of this pair of elements is the corresponding congruence subgroup $\sqb(11 + 5 \sqrt{2})$.  This congruence subgroup can also be generated by a single element; however, the minimal trace that we have found of such a generator is $2299 + 1625 \sqrt{2}$, for instance for the generator
\[
\beta \alpha \beta \alpha^{{{-1}}} \beta^{{{-1}}} \alpha \beta^{{{-1}}} \alpha
\beta \alpha^{{{-1}}} \beta \alpha^{{{-1}}} \beta \alpha^{{{-1}}} \beta
\alpha^{{{-1}}} \beta^{{{-1}}} \alpha^{{{-1}}} \beta \alpha \beta^{{{-1}}} \alpha
\beta^{{{-1}}} \alpha^{{{-1}}} \beta \alpha^{{{-1}}} \beta \alpha^{{{-1}}} \beta
\alpha^{{{-1}}}.
\]

The congruence subgroup has a presentation with $2g$ generators and a single relation of length $4g$, for
the expected $g = 7456 = 47(47^2 - 1)/48 + 1$.

For its Bolza twin, we were unable to find any element whose normal
closure is the entire congruence subgroup~$\sqb(11-5\sqrt{2})$.  This
congruence subgroup again has a presentation with $2g$ generators and
a single relation of length $4g$; however, the normal closure of each
of these generators have index at least $3$ in $\sqb(11 - 5\sqrt{2})$.
The smallest trace among the $2g = 14912$ generators is $\pm (633 +
449 \sqrt{2})$, which is obtained for eighteen of them.  We note that
\magma{} was unable the determine the index in $\sqb (11 - 5\sqrt{2})$
of the normal closure of all eighteen of these generators; this index
is likely to be very large or infinite.  However, the congruence
subgroup can be normally generated by the two elements
\[
\beta \alpha \beta \alpha^{-1} \beta \alpha \beta^{-1} \alpha^{-1}
\beta\alpha^{-1} \beta^{-1} \alpha^{-1} \beta \alpha^{-1} \beta
\alpha^{-1} \beta^{-1} \alpha^{-1} \beta \alpha^{-1} \beta^{-1} \alpha
\beta \alpha^{-1} \beta \alpha \beta \alpha^{-1} \beta \alpha^{-1}
\]
and
\[
\beta \alpha \beta \alpha^{-1} \beta \alpha^{-1} \beta \alpha^{-1}
\beta^{-1} \alpha \beta^{-1} \alpha^{-1} \beta^{-1} \alpha \beta
\alpha \beta^{-1} \alpha^{-1} \beta^{-1} \alpha \beta \alpha^{-1}
\beta^{-1} \alpha \beta \alpha \beta^{-1} \alpha \beta  \alpha^{-1}
\beta \alpha^{-1},
\]
which are congruent to $1$ modulo $(11 - 5 \sqrt{2})$ and have traces
$633 + 449 \sqrt{2}$ and the next smallest $-(1527 + 1080 \sqrt{2})$,
respectively.

\subsection{Summary of results}

To summarize, we collect some of the results presented above in a
table.  Each line of the table corresponds to a prime ideal $I
\normali O_K = \Z [ \sqrt{2}]$ dividing a rational prime $p$ that
splits in $K = \Q(\sqrt{2})$.  We present the lowest trace discovered
by our \texttt{magma} computations of a non-trivial element in the
congruence subgroup $\sqb(I)$, as well as the decimal expansion of
this candidate for the lowest trace, rounded to the nearest
thousandth.  For comparison, the rightmost column displays the
theoretical lower bound $N(I)^2 / 4 - 2$ for the trace.

For some ideals, such as $I = (11 - 5 \sqrt{2})$, we find elements
whose traces are remarkably close to the theoretical bound.  For other
ideals, our experimental results are not as close to the theoretical
bound; we ask whether elements of lower trace exist that could be
discovered by other methods.

\begin{center}
\noindent \begin{minipage}{\textwidth}
\begin{longtable}{| r | r || r | r || r |} \hline
$I$ & $N(I)$ & lowest trace & & $N(I)^2 / 4 - 2$ \\ \hline \hline
$(1 + 2\sqrt{2})$ & $7$ & $ 7 + 4\sqrt{2}$ & $12.657$ & $10.25$ \\ \hline
$(1 - 2 \sqrt{2})$ & $7$ & $9 + 6 \sqrt{2}$ & $17.485$ & $10.25$ \\ \hline
$(1 - 3 \sqrt{2})$ & $17$ & $75 + 53 \sqrt{2}$ & $149.953$ & $70.25$ \\ \hline
$(1 + 3 \sqrt{2})$ & $17$ & $79 + 56 \sqrt{2}$ & $158.196$ & $70.25$ \\ \hline
$(5 - \sqrt{2})$ & $23$ & $91 + 65 \sqrt{2}$ & $182.924$ & $130.25$ \\ \hline
$(5 + \sqrt{2})$ & $23$ & $119 + 84 \sqrt{2}$ & $237.794$ & $130.25$ \\ \hline
$(9 + 5 \sqrt{2})$ & $31$ & $129 + 90 \sqrt{2}$ & $256.279$ & $238.25$ \\ \hline
$(9 - 5 \sqrt{2})$ & $31$ & $153 + 109 \sqrt{2}$ & $307.149$ & $238.25$ \\ \hline
$(7 + 2 \sqrt{2})$ & $41$ & $281 + 198 \sqrt{2}$ & $561.014$ & $418.25$ \\ \hline
$(7 - 2 \sqrt{2})$ & $41$ & $295 + 208 \sqrt{2}$ & $589.156$ & $418.25$ \\ \hline
$(7 + \sqrt{2})$ & $47$ & $499 + 353 \sqrt{2}$ & $998.217$ & $550.25$ \\ \hline
$(7 - \sqrt{2})$ & $47$ & $529 + 374 \sqrt{2}$ & $1057.916$ & $550.25$ \\ \hline
$(11 - 5 \sqrt{2})$ & $71$ & $633 + 449 \sqrt{2}$ & $1267.982$ & $1258.25$ \\ \hline
$(11 + 5 \sqrt{2})$ & $71$ & $951 + 672 \sqrt{2}$ & $1901.352$ & $1258.25$ \\ \hline
\end{longtable}
\end{minipage}

\end{center}

\section*{Acknowledgments}

We are grateful to M.~Belolipetsky and J.~Voight for helpful comments.
M. Katz is partially supported by ISF grant 1517/12.  M.~M. Schein is
partially supported by GIF grant 1246/2014.  U. Vishne is partially
supported by a BSF grant 206940 and an ISF grant 1207/12.

\end{document}